\documentclass[12pt]{amsart}
\usepackage{amsmath, amssymb}
\usepackage{mathrsfs}
\newcommand{\no}[1]{#1}
\renewcommand{\no}[1]{}
\no{\usepackage{times}\usepackage[subscriptcorrection, slantedGreek, nofontinfo]{mtpro}
	\renewcommand{\Delta}{\upDelta}}
\usepackage{color}
\usepackage{enumerate,comment}
\usepackage{mathrsfs}
\usepackage[ocgcolorlinks,linkcolor=blue]{hyperref}

\usepackage{bm}
\usepackage{bbm}
\usepackage{url}

\newtheorem{Thm}{Theorem}[section]
\newtheorem{prop}{Proposition}[section]
\newtheorem{lem}{Lemma}[section]
\newtheorem{definition}{Definition}[section]
\newtheorem{corollary}{Corollary}[section]

\newtheorem{assumption}{Assumption}[section]
\newtheorem{example}{Example}[section]

\theoremstyle{remark}
\newtheorem{remark}{Remark}[section]

\newcommand{\pd}{\partial}

\newcommand{\R}{{\mathbb R}}

\renewcommand{\leq}{\leqslant}
\renewcommand{\geq}{\geqslant}

\def\beq{\begin{equation}}
	\def\eeq{\end{equation}}
\newcommand{\bea}{\begin{eqnarray}}
	\newcommand{\eea}{\end{eqnarray}}
\newcommand{\beas}{\begin{eqnarray*}}
	\newcommand{\eeas}{\end{eqnarray*}}

\providecommand{\abs}[1]{\left\lvert#1\right\rvert}
\providecommand{\norm}[1]{\left\lVert#1\right\rVert}

\newcommand{\ccdot}{\,\cdot\,}
\newcommand{\s}{\hspace{0.5pt}}

\numberwithin{equation}{section}

\title[Inverse problems for a semilinear reaction--diffusion equation]
{
	On determining and breaking the gauge class in inverse problems for reaction-diffusion equations 
}

\author[Y. Kian]{Yavar Kian}
\address{Aix Marseille Univ, Université de Toulon, CNRS, CPT, Marseille, France}
\curraddr{}
\email{yavar.kian@univ-amu.fr}


\author[T. Liimatainen]{Tony Liimatainen}
\address{Department of Mathematics and Statistics, University of Helsinki, Helsinki, Finland}
\curraddr{}
\email{tony.liimatainen@helsinki.fi}

\author[Y.-H. Lin]{Yi-Hsuan Lin}
\address{Department of Applied Mathematics, National Yang Ming Chiao Tung University, Hsinchu, Taiwan}
\curraddr{}
\email{yihsuanlin3@gmail.com}

\begin{document}
	
	\begin{abstract}
		We investigate an inverse boundary value problem of determination of a nonlinear law for reaction-diffusion processes, which are modeled by general form semilinear parabolic equations. 
		We do not assume that any solutions to these equations are known a priori, in which case the problem has a well known gauge symmetry. 
		We determine, under additional assumptions, the semilinear term up to this symmetry in a time-dependent anisotropic case modeled on Riemannian manifolds, and for partial data measurements on $\R^n$.  
		
		Moreover, we present cases where it is possible to exploit  the nonlinear interaction to break the gauge symmetry. This leads to full determination results of the nonlinear term. As an application, we show that it is possible to give a full resolution to classes of inverse source problems of determining a source term and nonlinear terms simultaneously. This is in strict contrast to inverse source problems for corresponding linear equations, which always have the gauge symmetry. We also consider a Carleman estimate with boundary terms based on intrinsic properties of parabolic equations. 

		\medskip
		

	\end{abstract}
	\maketitle
	
	\tableofcontents


	\section{Introduction}\label{sec_intro}

	In this article we address the following question: "Is it possible to determine the non-linear law of a reaction-diffusion process by applying sources and measuring the corresponding flux on the boundary of the domain of diffusion?". Mathematically, 
	the question can be stated as a problem of determination of a general semilinear term appearing in a semilinear parabolic equation from boundary measurements.
	
	Let us explain the problem  precisely. Let $T>0$ and let $\Omega\subset \R^n$, with $n\geq 2$, be a bounded and connected domain with a smooth boundary. 
	Let $a:=(a_{ik})_{1 \leq i,k \leq n} \in  C^\infty([0,T]\times\overline{\Omega};\R^{n\times n})$
	be symmetric matrix field, 
	$$ 
	a_{ik}(t,x)=a_{ki}(t,x),\ x \in \Omega,\ i,k = 1,\ldots,n, \ (t,x)\in [0,T]\times\overline{\Omega},
	$$
	which fulfills the following ellipticity condition: there exists a constant
	$c>0$ such that
	\begin{align}\label{ell}
		\sum_{i,k=1}^n a_{ik}(t,x) \xi_i \xi_k \geq c |\xi|^2, \quad 
		\mbox{for each $(t,x)\in [0,T]\times\overline{\Omega},\ \xi=(\xi_1,\ldots,\xi_n) \in \R^n$}.
	\end{align}
	We define elliptic operators $\mathcal A(t)$, $t\in [0,T]$, in divergence form  by
	$$ 
	\mathcal A(t) u(t,x) :=-\sum_{i,k=1}^n \partial_{x_i} 
	\left( a_{ik}(t,x) \partial_{x_k} u(t,x) \right),\  x\in\overline{\Omega},\ t\in[0,T]. 
	$$ 
	Throughout the article, we set 
	$$
	Q:= (0,T) \times \Omega, \quad \Sigma := (0,T) \times \partial\Omega
	$$
	and we refer to $\Sigma$ as the lateral boundary. Let us also fix
	$\rho \in C^\infty([0,T]\times\overline{\Omega};\R_+)$ and and $b\in C^\infty([0,T]\times\overline{\Omega}\times\R)$. Here $\R_+=(0,+\infty)$.
	Then, we consider  the following initial boundary value problem (IBVP in short):
	\begin{align}\label{eq1}
		\begin{cases}
			\rho(t,x) \partial_t u(t,x)+\mathcal A(t)  u(t,x)+ b(t,x,u(t,x))  =  0, & (t,x)\in 
			Q,\\
			u(t,x)  =  f(t,x), & (t,x) \in \Sigma, \\  
			u(0,x)  =  0, & x \in \Omega.
		\end{cases}
	\end{align}
	The parabolic \emph{Dirichlet-to-Neumann} map (DN map in short) is formally defined by 
	$$
	\mathcal N_{b}:   f\mapsto \left.\partial_{\nu(a)} u \right|_\Sigma,
	$$
	where $u$ is the solution to \eqref{eq1}. Here the conormal derivative $\partial_{\nu(a)}$ associated to the coefficient $a$  is defined by
	$$\partial_{\nu(a)}v(t,x):=\sum_{i,k=1}^na_{ik}(t,x)\partial_{x_k}v(t,x)\nu_i(x),\quad (t,x)\in\Sigma,$$
	where $\nu=(\nu_1,\ldots,\nu_n)$ denotes the outward unit normal vector of $\partial\Omega$ with respect to the Euclidean $\R^n$ metric. The solution $u$ used in the definition of the DN map $\mathcal{N}_b$ is unique in a specific sense so that there is no ambiguity in the definition of $\mathcal{N}_b$. For this fact and a rigorous definition of the DN map we refer to  to Section \ref{sec_prelim}.  We write simply $\partial_{\nu}=\partial_{\nu(a)}$ in the case $a$ is the $n\times n$ identity matrix $\mathrm{Id}_{\R^{n\times n}}$. 
	The inverse problem we study is the following.
	\begin{itemize}
		\item \textbf{Inverse problem (IP):} Can we recover the semilinear term $b$ from the knowledge of the parabolic Dirichlet-to-Neumann map $\mathcal N_{b}$?
	\end{itemize}
	
	Physically,  reaction diffusion equations of the form \eqref{eq1} describe several classes of diffusion processes with applications in chemistry, biology, geology, physics  and ecology. This includes the spreading of biological populations  \cite{Fi}, the Rayleigh-B\'enard convection  \cite{NW} or models appearing in  combustion theory  \cite{volpert2014elliptic, ZF}. The inverse problem (IP) is equivalent to  the determination of an underlying physical law of a diffusion process, described by the nonlinear expression  $b$ in \eqref{eq1}, by applying different  sources (e.g. heat sources) and measuring the corresponding flux at the lateral  boundary $\Sigma$. The information extracted from this way is encoded into the  DN map  $\mathcal N_{b}$.
	
	These last decades, problems of parameter identification in nonlinear partial differential equations  have received   a large interest in the mathematical community. Among the different formulation of these inverse problems, the determination of a nonlinear law  is one of the most challenging from the sever ill-posedness and nonlinearity of the problem. For diffusion equations, one of the first results in that direction can be found in \cite{Is1}. Later on this result was improved by \cite{CK}, where the stability issue was also considered. To the best of our knowledge, the most general and complete 
	result known so far about the determination of a semilinear term of the form $b(t,x,u)$ depending simultaneously on the time variable $t$, the space variable $x$ and the solution of the equation $u$ from knowledge of the parabolic DN map $\mathcal N_{b}$ can be found in \cite{KiUh}.
	Without being exhaustive, we mention  the works of \cite{IS3,CY,COY} devoted to the determination of semilinear terms depending only on the solution and the determination of quasilinear terms addressed in \cite{CK1,EPS,FKU}. Finally, we mention    the works of 
	\cite{FO20,KLU,LLLS,LLLS2019partial, LLST2022inverse,KU,KU2019partial,FLL2021inverse,harrach2022simultaneous,KKU,CFKKU,FKU,LL2020inverse,lin2020monotonicity,KiUh,LL22inverse,lai2019global}
	devoted to similar problems for elliptic and hyperbolic equations. Moreover, in the recent works \cite{LLLZ2021simultaneous,LLL2021determining}, the authors investigated simultaneous determination problems of coefficients and initial data for both parabolic and hyperbolic equations.

	Most of the above mentioned results concern the inverse problem (IP) under the assumption that the semilinear term $b$ in \eqref{eq1} satisfies the  condition 
	\begin{equation}\label{con}
		b(t,x,0)=0,\quad (t,x)\in Q.
	\end{equation}
	This condition implies that \eqref{eq1} has at least one known solution, the trivial solution. 
	In the same spirit, for any constant $\lambda\in \R$, the condition $
	b(t,x,\lambda)=0$ for $(t,x)\in Q$
	implies that the constant function $(t,x)\mapsto\lambda$ is a solution of the equation $\rho(t,x) \partial_t u+\mathcal A(t)  u+ b(t,x,u)  =  0$ in $Q$. In the present article, we treat the determination of general class of semilinear terms for which the condition \eqref{con} may not be fulfilled.
	In this case, the problem (IP) is even more challenging since no solutions  of \eqref{eq1} may not be known a priori. In fact, as observed in \cite{Sun2} for elliptic equations, there is an obstruction, a \emph{gauge symmetry}, to the determination of $b$ from the knowledge of $\mathcal N_b$. We demonstrate the gauge symmetry first in the form of an example. 
	\begin{example}\label{rmk: example}
		Let us consider the inverse problem (IP) for the simplest linear case 
		\begin{align}\label{linear u ex}
			\begin{cases}
				\partial_t u -\Delta u =b_0  &\text{ in }Q, \\
				u=f& \text{ on }\Sigma,\\
				u(0,x)=0 &\text{ in }\Omega,
			\end{cases}
		\end{align}
		where the aim is to recover an unknown source term $b_0=b_0(x,t)$. Here $\Delta$ is the Laplacian, but it could also be replaced by a more general second order elliptic operator whose coefficients are known. 
		Let us consider a function $\varphi\in  C^{\infty}([0,T]\times\overline{\Omega})$, which satisfies  $\varphi\not\equiv0$,  $\varphi(0,x)=0$ for $x\in\Omega$ and $\varphi=\partial_{\nu} \varphi=0$ on $\Sigma$, where $\partial_\nu \varphi$ denotes the Neumann derivative of $\varphi$ on $\Sigma$.
		
		
		Then the function $\tilde u:=u+\varphi$ satisfies 
		\begin{align}\label{linear tilde u ex}
			\begin{cases}
				\partial_t \tilde u -\Delta \tilde u =b_0+(\partial_t -\Delta)\varphi  &\text{ in }Q, \\
				\tilde u=f& \text{ on }\Sigma,\\
				\tilde u(0,x)=0 &\text{ in }\Omega.
			\end{cases}
		\end{align}
		Since $u$ and $\tilde u$ have the same initial data at $t=0$ and Cauchy data on $\Sigma$, we see that the DN maps of \eqref{linear u ex} and \eqref{linear tilde u ex} are the same.
		However, by unique continuation properties for parabolic equations \emph{(}see e.g. \cite[Theorem 1.1]{SaSc}\emph{)}, the conditions  $\varphi=\partial_{\nu} \varphi=0$ on $\Sigma$ and $\varphi\not\equiv0$ imply that  $(\partial_t -\Delta)\varphi\not\equiv0$, and it follows that $b_0+(\partial_t -\Delta)\varphi\neq b_0$.   Consequently, the inverse problem (IP) can not be uniquely solved.
	\end{example}

	In what follows, we assume that $\alpha>0$ and refer to the definitions of various function spaces that will show up to Section \ref{sec_prelim}. Let us  describe the gauge symmetry, or gauge invariance, of the inverse problem (IP) in detail.  For this, let a function $\varphi\in   C^{1+\frac{\alpha}{2},2+\alpha}([0,T]\times\overline{\Omega})$ again satisfy 
	\begin{align}\label{gauge1}
		\varphi(0,x)=0,\quad x\in\Omega,\quad \varphi(t,x)=\partial_{\nu(a)} \varphi(t,x)=0,\quad (t,x)\in\Sigma,
	\end{align}
	and consider the mapping $S_\varphi$ from $ C^\infty(\R;C^{\frac{\alpha}{2},\alpha}([0,T]\times\overline{\Omega}))$ into itself 
	defined by 
	\begin{align}\label{gauge2}
		S_\varphi b(t,x,\mu)=b(t,x,\mu+\varphi(t,x))+\rho(t,x) \partial_t \varphi(t,x)+\mathcal A(t)  \varphi(t,x),\quad (t,x,\mu)\in Q\times\R.
	\end{align}
	As in the example above, one can easily check that $\mathcal N_{b}=\mathcal N_{S_\varphi b}$.

	In view of this obstruction, the inverse problem (IP) should be reformulated as a problem about determining the semilinear term up to the gauge symmetry described by \eqref{gauge2}. We note that \eqref{gauge2} implies  an equivalence relation for functions in $ C^\infty(\R;C^{\frac{\alpha}{2},\alpha}([0,T]\times\overline{\Omega}))$. Corresponding equivalence classes will be called gauge classes:
	\begin{definition}[Gauge class]
		We say that two nonlinearities $b_1,b_2\in C^\infty(\R;C^{\frac{\alpha}{2},\alpha}([0,T]\times\overline{\Omega})$ are in the same gauge class, or equivalent up to a gauge, if there is $\varphi\in   C^{1+\frac{\alpha}{2},2+\alpha}([0,T]\times\overline{\Omega})$ satisfying \eqref{gauge1} 
	such that
	\begin{equation}\label{eq:gauge_equivalence}
		b_1=S_\varphi b_2.
	\end{equation} 
	Here $S_\varphi$ is as in \eqref{gauge2}. 
	In the case we consider a partial data inverse problem, where the normal derivative of solutions is assumed to be known only on $(0,T)\times \tilde \Gamma$, with $\tilde \Gamma$ open in $\partial\Omega$, we assume  $\partial_{\nu(a)} \varphi=0$ only on $(0,T)\times \tilde \Gamma$ in \eqref{gauge1}.
	
\end{definition}
Using the above definition of gauge the invariance and taking into account the obstruction described above, we reformulate the inverse problem (IP) as follows.
\begin{itemize}
	\item \textbf{Inverse problem (IP1):} Can we determine the gauge class of the semilinear term $b$ from the full or a partial knowledge  of the parabolic DN map $\mathcal N_{b}$?
\end{itemize}

There is a natural additional question raised by (IP1), namely
\begin{itemize}
	\item \textbf{Inverse problem (IP2):}  When does the gauge invariance break leading to the full resolution of problem (IP)?
\end{itemize}
One can easily check  that it is possible to give a positive answer to problem (IP2)  when \eqref{con} is fulfilled. Nevertheless, as observed in the recent work \cite{liimatainen2022uniqueness}, the resolution of problem (IP2) is not restricted to such a situation. The work \cite{liimatainen2022uniqueness} provided the first examples (in an elliptic setting) how to use nonlinearity as a tool to break the gauge invariance of (IP). 

Let us also remark that, following Example \ref{rmk: example}, when $u\mapsto b(\ccdot,u)$ is affine, corresponding to the case where the equation \eqref{eq1} is linear and has a source term, there is no hope to break the gauge invariance \eqref{gauge2}  and the response to (IP2) is in general negative. As will be observed in this article this is no longer the case for various classes of nonlinear terms $b$. We will present cases where we will be able to solve (IP) uniquely. These cases present new instances where nonlinear interaction can be helpful in inverse problems. Nonlinearity has earlier been observed to be a helpful tool by many authors in different situations such as in partial data inverse problems and in anisotropic inverse problems on manifolds, see e.g. \cite{KLU,FO20,LLLS,LLLS2019partial, LLST2022inverse,KU2019partial,KU,FLL2021inverse}. 

In the present article we will address both problems (IP1) and (IP2). We will start by considering the problem (IP1) for nonlinear terms, which are quite general. Then, we will exhibit several general situations where the gauge invariance breaks and give an answer to (IP2).

We mainly restrict our analysis to semilinear terms $b$ subjected to the condition that the map $u\mapsto b(\ccdot,u)$ is analytic (the $(t,x)$ dependence is of $b(t,x,\mu)$ will not be assumed to be analytic). The restriction to this class of nonlinear terms is motivated by the study of the challenging problems (IP1) and (IP2) in this article. Indeed, even when condition \eqref{con} is fulfilled, the problem (IP) is still open in general for semilinear terms that are not subjected to our analyticity condition (see \cite{KiUh} for the most complete results known so far for this problem). Results for semilinear elliptic equations have also been proven for cases when the nonlinear terms are roughly speaking globally Lipschitz (see e.g.  \cite{Is4,Is5}).  For this reason, our assumptions seem reasonable for tackling problems (IP1) and (IP2) and giving the first answers to these challenging problems. Note also that the linear part of \eqref{eq1} will be associated with a general class of linear parabolic equations with variable time-dependent coefficients.  Consequently, we will present results also for linear equations with full and partial data measurements.

\section{Main results}\label{sec_results}	
In this section, we will first introduce some preliminary definitions and results required for the rigorous formulation of our problem (IP). Then, we will state our main results for problems (IP1) and (IP2). 

\subsection{Preliminary properties}\label{sec_prelim}

From now on, we fix $\alpha\in(0,1)$ and we denote   by
$C^{\frac{\alpha}{2},\alpha}([0,T]\times X)$, with $X=\overline{\Omega}$ or $X=\partial\Omega$,  the set of functions $h$ lying in  $ C([0,T]\times  X)$ satisfying
$$[h]_{\frac{\alpha}{2},\alpha}=\sup\left\{\frac{|h(t,x)-h(s,y)|}{(|x-y|^2+|t-s|)^{\frac{\alpha}{2}}}:\, (t,x),(s,y)\in [0,T]\times X,\ (t,x)\neq(s,y)\right\}<\infty.$$
Then we define the space $ C^{1+\frac{\alpha}{2},2+\alpha}([0,T]\times X)$ as the set of functions $h$ lying in 
$$C([0,T];C^2(X))\cap C^1([0,T];C(X))$$ 
such that $$\partial_th,\partial_x^\beta h\in C^{\frac{\alpha}{2},\alpha}([0,T]\times X),\quad \beta\in(\mathbb N\cup\{0\})^n,\ |\beta|=2.$$
We consider these spaces with their usual norms and we refer to \cite[pp. 4]{Ch} for more details. 	Let us also introduce the space
$$\mathcal K_0:=\left\{h\in C^{1+\frac{\alpha}{2},2+\alpha}([0,T]\times\partial\Omega):\, h(0,\ccdot)=\partial_th(0,\ccdot)= 0\right\}.$$
If $r>0$ and  $h\in \mathcal K_0$, we denote by 
\begin{align}\label{D-data ball}
	\mathbb B(h,r):= \left\{ g\in \mathcal{K}_0:\, \norm{g-h}_{C^{1+\frac{\alpha}{2},2+\alpha}([0,T]\times\partial\Omega)}<r \right\},
\end{align}
the ball centered at $h$ and of radius $r$ in the space $\mathcal K_0$. We assume also that $b$ fulfills the following condition
\begin{align}\label{b}
	b(0,x,0)=0,\quad x\in\partial\Omega.
\end{align}

In this article we assume that there exists $f_0\in \mathcal K_0$ such that \eqref{eq1} admits a unique solution when $f=f_0$. According to \cite[Theorem 6.1, pp. 452]{LSU}, \cite[Theorem 2.2, pp. 429]{LSU}, \cite[Theorem 4.1, pp. 443]{LSU}, \cite[Lemma 3.1, pp. 535]{LSU} and \cite[Theorem 5.4, pp. 448]{LSU}, the problem \eqref{eq1} is well-posed  for any $f=f_0\in \mathcal K_0$ if for instance  there exist $c_1,c_2\geq0$ such that the semilinear term $b$ satisfies the following sign condition 
$$ b(t,x,\mu)\mu\geq -c_1\mu^2-c_2,\quad t\in[0,T],\ x\in\overline{\Omega},\ \mu\in\R.$$

The unique existence of solutions  of \eqref{eq1} for some $f\in \mathcal K_0$ is not restricted to such situation. Indeed, assume that there exists $\psi\in  C^{1+\frac{\alpha}{2},2+\alpha}([0,T]\times\overline{\Omega})$ satisfying 
\begin{align*}
	\begin{cases}
		\rho(t,x) \partial_t \psi(t,x)+\mathcal A(t)  \psi(t,x)  =  0 & \text{ in } 
		Q,\\ 
		\psi(0,x)  =  0 & \text{ for }x \in \Omega,
	\end{cases}
\end{align*}
such that $b(t,x,\psi(t,x))=0$ for $(t,x)\in[0,T]\times\overline{\Omega}$.  Then, one can easily check that  \eqref{eq1} admits a unique solution when $f=\psi|_{\Sigma}$. Moreover, applying Proposition \ref{p1}, we deduce that  \eqref{eq1} will be well-posed
when $f\in \mathcal K_0$ is sufficiently close to $\psi|_{\Sigma}$ in the sense of $C^{1+\frac{\alpha}{2},2+\alpha}([0,T]\times\partial\Omega)$.

As  will be shown in Proposition \ref{p1}, the existence of $f_0\in \mathcal K_0$ such that \eqref{eq1} admits a solution  when $f=f_0$  implies that there exists $\epsilon>0$, depending on  $a$, $\rho$, $b$, $f_0$, $\Omega$, $T$, such that, for all $f\in \mathbb B(f_0,\epsilon)$,  the problem \eqref{eq1} admits a solution $u_f\in  C^{1+\frac{\alpha}{2},2+\alpha}([0,T]\times\overline{\Omega})$, which is unique in a sufficiently small neighborhood of the solution of \eqref{eq1}  with boundary value $f=f_0$.
Using these  properties, we can define the parabolic DN map 
\begin{align}\label{DN map}
	\mathcal N_{b}:\mathbb B(f_0,\epsilon)\ni f\mapsto \partial_{\nu(a)} u_f(t,x),\quad (t,x)\in\Sigma.
\end{align}

\subsection{Resolution of (IP1)}\label{sec 2.2}
We present our first main results about recovering the gauge class of a semilinear term from the corresponding DN map.

We consider $\Omega_1$ to be an open bounded, smooth and connected subset of $\R^n$ such that $\overline{\Omega}\subset\Omega_1$. We extend $a$ and $\rho$ into functions defined smoothly on $[0,T]\times\overline{\Omega}_1$ satisfying $\rho>0$ and condition \eqref{ell} with $\Omega$ replaced by $\Omega_1$. For all $t\in[0,T]$, we set also 
$$
g(t):=\rho(t,\ccdot)a(t,\ccdot)^{-1},
$$ 
and we consider the compact  Riemannian manifold with boundary $(\overline{\Omega}_1,g(t))$. 

\begin{assumption}\label{assump simple}
	Throughout this article, we assume that $\left(\overline{\Omega}_1,g(t)\right)$ is a simple Riemannian manifold  for all $t\in[0,T]$. That is, we assume that for any point $x\in \overline{\Omega}_1$ the exponential map $\exp_x$ is a diffeomorphism from some closed neighborhood of $0$ in $T_x\s \overline{\Omega}_1$ onto $\overline{\Omega}_1$ and $\partial\Omega_1$ is strictly convex. 
\end{assumption}

From now on, for any Banach space $X$,  we denote by $\mathbb A(\R;X)$ the set of analytic functions on $\R$ as maps taking values in $X$. That is, for any $b\in \mathbb A(\R;X)$ and $\mu\in \R$, $b$ has convergent $X$-valued Taylor series on a neighborhood of $\mu$ . 

\begin{Thm}\label{t1} Let $a:=(a_{ik})_{1 \leq i,k \leq n} \in C^\infty([0,T]\times\overline{\Omega};\R^{n\times n})$ satisfy \eqref{ell} and 
	$\rho \in C^\infty([0,T]\times\overline{\Omega};\R_+)$. Let $b_j\in\mathbb A(\R;C^{\frac{\alpha}{2},\alpha}([0,T]\times\overline{\Omega}))$, which satisfies \eqref{b} as $b=b_j$, for $j=1,2$. We also assume that there exists $f_0\in \mathcal K_0$ such that problem \eqref{eq1}, with $f=f_0$ and $b=b_j$, admits a unique solution for $j=1,2$. 
	Then,  the condition
	\begin{align}\label{t1c}
		\mathcal N_{b_1}(f)=\mathcal N_{b_2}(f),\quad f\in \mathbb B(f_0,\epsilon), 
	\end{align}
	implies that there exists  $\varphi\in   C^{1+\frac{\alpha}{2},2+\alpha}([0,T]\times\overline{\Omega})$ satisfying \eqref{gauge1} such that
	\begin{align}\label{t1d}
		b_1=S_\varphi b_2
	\end{align}
	with $S_\varphi$ the map defined by \eqref{gauge2}.
\end{Thm}

\begin{remark}\label{rmk_taylor expansion}
	Let us remark what can be recovered by our methods without the assumption of analyticity  of $b$ in the $\mu$-variable. In this case we can only recover the Taylor series of $b$ in $\mu$-variable at shifted points. Indeed, assume as in Theorem \ref{t1} with the exception that the nonlinearities $b_1$ and $b_2$ are not analytic in the $\mu$-variable, and fix $(t,x)$. In this case, by inspecting the proof of Theorem \ref{t1}(see Section \ref{sec_proof_thm}), we can show that 
	\[
	\partial_\mu^k\s b_1\left(t,x,u_{1,0}(t,x)\right)=\partial_\mu^k\s b_2\left(t,x,u_{2,0}(t,x)\right),\quad \text{ for any } k\in \mathbb N,
	\]
	where $u_{j,0}(t,x)$, for $j=1,2$, is the solution to \eqref{eq1} with coefficient $b=b_j$ and boundary value $f=f_0\in \mathcal{K}_0$. Thus we see that the formal Taylor series of $b_1(t,x,\ccdot)$ at $u_{1,0}(t,x)$ is that of $b_2(t,x,\ccdot)$ shifted by $u_{2,0}(t,x)-u_{1,0}(t,x)$, which is typically non-zero as we do not assume that we know any solutions to \eqref{eq1} a priori. 
	
	By assuming analyticity in Theorem \ref{t1} we are able to connect the Taylor series of $b_1$ and $b_2$ in the $\mu$-variable at different points, which leads to \eqref{t1d} in the end. This is one motivation for the analyticity assumption. Note that we do not assume analyticity in the other variables.
\end{remark}

For our second result, let us consider a \emph{partial data result} when $\mathcal A(t)=-\Delta$ (that is, $a=\text{Id}_{\s \R^{n\times n}}$ is an $n\times n$ identity matrix) and $\rho\equiv1$. More precisely, consider  the front and back sets of $\partial \Omega$
$$\Gamma_\pm(x_0):=\left\{x\in\partial\Omega:\, \pm(x-x_0)\cdot\nu(x)\geq0 \right\}$$
with respect to a source $x_0\in\R^n\setminus\overline{\Omega}$.
Then, our second main result is stated as follows: 

\begin{Thm}\label{t5} For $n\geq3$ and $\Omega$ simply connected, let $a=(a_{ik})_{1 \leq i,k \leq n}= \mathrm{Id}_{\,\R^{n\times n}}$ and
	$\rho \equiv1$. Let $b_j\in\mathbb A(\R;C^{\frac{\alpha}{2},\alpha}([0,T]\times\overline{\Omega}))$, which satisfies \eqref{b} as $b=b_j$, for $j=1,2$. We also assume that there exists $f_0\in \mathcal K_0$ such that problem \eqref{eq1}, with $f=f_0$ and $b=b_j$, admits a unique solution, for $j=1,2$. Fix $x_0\in\R^n\setminus\overline{\Omega}$ and consider $\tilde{\Gamma}$ a neighborhood of $\Gamma_-(x_0)$ on $\partial\Omega$.
	Then,  the condition
	\begin{align}\label{t5a}
		\mathcal N_{b_1}f(t,x)=\mathcal N_{b_2}f(t,x),\quad (t,x)\in (0,T)\times  \tilde{\Gamma}, \text{ for any } f\in\mathbb B(f_0,\epsilon),
	\end{align} 
	implies that there exists $\varphi\in   C^{1+\frac{\alpha}{2},2+\alpha}([0,T]\times\overline{\Omega})$ satisfying 
	\begin{align}\label{t5b}
		\varphi(0,x)=0,\quad x\in\Omega,\quad \varphi(t,x)=0,\quad (t,x)\in\Sigma,\\
		\label{t5bb}\partial_{\nu} \varphi(t,x)=0,\quad (t,x)\in (0,T)\times \tilde{\Gamma}
	\end{align}
	such that 
	\begin{align}\label{t1d_new}
		b_1=S_\varphi b_2.
	\end{align}
\end{Thm}
We will be able to be break the gauge condition $b_1=S_\varphi b_2$ in Theorems \ref{t1} and \ref{t5} in various cases. We present these results separately in the next section.

\subsection{Breaking the gauge in the sense of (IP2)} 
In several situation, the gauge class \eqref{t1d_new} could be broken  and one can fully determine the semilinear term $b$ in \eqref{eq1} from its parabolic DN map. We will present below classes of nonlinearities when such phenomenon occurs. We start by considering general elements of $b\in\mathbb A(\R;C^{\frac{\alpha}{2},\alpha}([0,T]\times\overline{\Omega}))$  for which the gauge invariance \eqref{t1d_new} breaks.

\begin{corollary}\label{c1} Let the conditions of Theorem \ref{t1} be fulfilled and assume that 
	there exists $\kappa\in C^{\frac{\alpha}{2},\alpha}([0,T]\times\overline{\Omega})$ such that 
	\begin{align}\label{c1a}
		b_1(t,x,\kappa(t,x))=b_2(t,x,\kappa(t,x)),\quad (t,x)\in [0,T]\times\overline{\Omega}.
	\end{align}
	Then,  the condition \eqref{t1c} implies that $b_1=b_2$.
	In the same way, assuming that the conditions of Theorem \ref{t5} are fulfilled, the condition \eqref{t5a} implies that $b_1=b_2$.

\end{corollary}

\begin{corollary}\label{c2} Let the conditions of Theorem \ref{t1} be fulfilled and assume that there exists $h\in C^{\alpha}(\overline{\Omega})$, $G\in C^{\frac{\alpha}{2},\alpha}([0,T]\times\overline{\Omega})$ and $\theta\in (0,T]$ satisfying the condition
	\begin{align}\label{c2a}
		\inf_{x\in\Omega}|G(\theta,x)|>0,
	\end{align}
	such that
	\begin{align}\label{c2b}
		b_1(t,x,0)-b_2(t,x,0)=h(x)G(t,x),\quad (t,x)\in [0,T]\times\overline{\Omega}.
	\end{align}
	Assume also that the solutions $u_{j,0}$ of \eqref{eq1}, $j=1,2$, with $f=f_0$ and $b=b_j$ satisfy the condition
	\begin{align}\label{c2c}
		u_{1,0}(\theta,x)=u_{2,0}(\theta,x),\quad x\in\Omega.
	\end{align}
	Then,  the condition \eqref{t1c} implies that $b_1=b_2$. Moreover, by assuming that the conditions of Theorem \ref{t5} are fulfilled, the condition \eqref{t5a} implies that $b_1=b_2$.
\end{corollary}

Now let us consider elements  $b\in\mathbb A(\R;C^{\frac{\alpha}{2},\alpha}([0,T]\times\overline{\Omega}))$ which are polynomials of the form
\begin{align}\label{pol}
	b(t,x,\mu)=\sum_{k=0}^N b_k(t,x)\mu^k,\quad (t,x,\mu)\in[0,T]\times\overline{\Omega}\times\R.
\end{align}
For this class of nonlinear terms we can prove the following.

\begin{Thm}\label{c3} Let the condition of Theorem \ref{t1} be fulfilled and assume that, for $j=1,2$, there exists $N_j\geq2$ such that
	\begin{align}\label{c3a}
		b_j(t,x,\mu)=\sum_{k=0}^{N_j} b_{j,k}(t,x)\mu^k,\quad (t,x,\mu)\in[0,T]\times\overline{\Omega}\times\R.
	\end{align}
	Let $\omega$  be an open subset   of $\R^n$ such that $\omega\subset \Omega$ and $J$ a dense subset of $(0,T)\times\omega$.
	We assume also that, for $N=\min(N_1,N_2)$, the following conditions 
	\begin{align}\label{c3b}
		\min \left(\left|(b_{1,N-1}-b_{2,N-1})(t,x)\right|, \, \sum_{j=1}^2\left|(b_{j,N}-b_{j,N-1})(t,x)\right|\right)=0,\ (t,x)\in  J,
	\end{align}
	\begin{align}\label{c3c}
		\left|b_{1,N}(t,x)\right|>0,\quad (t,x)\in J
	\end{align}
	\begin{align}\label{c3d}
		b_{1,0}(t,x)=b_{2,0}(t,x),\quad (t,x)\in(0,T)\times(\Omega\setminus{\overline{\omega}}),
	\end{align}
	hold true.
	Then the condition \eqref{t1c} implies that $b_1=b_2$. In addition, assuming that the conditions of Theorem \ref{t5} are fulfilled, condition \eqref{t5a} implies that $b_1=b_2$.
\end{Thm}

We make the following remark. 

\begin{remark}
	\phantom{} 
	
	\begin{itemize}
		\item[(i)] The preceding theorem in particularly says that the inverse source problem of recovering a source function $F$ from the DN map of
		\[
		\partial_t u-\Delta u + u^2=F
		\]
		is uniquely solvable. This is in strict contrast to the inverse source problem for the linear equation $\partial_t u-\Delta u =F$ that always a gauge invariance explained in Example \ref{rmk: example}: the sources $F$ and $\widetilde{F}:=F+\partial_t\varphi -\Delta \varphi$ in $Q$ have the same DN map. Here the only restrictions for $\varphi$ are given in \eqref{gauge1} and thus typically $F\neq \widetilde F$.
		
		
		\item[(ii)]  Inverse source problems for semilinear elliptic equations were studied in \cite{liimatainen2022uniqueness}. There it was shown that if in the  notation of the corollary $b_{1,N-1}=b_{2,N-1}$  and $b_{1,N}\neq 0$ in $\Omega$ (so that \eqref{c3b} and \eqref{c3c} hold), then the gauge breaks. With natural replacements, the corollary generalizes \cite[Corollary 1.3]{liimatainen2022uniqueness} in the elliptic setting. This can be seen by inspecting its proof.
		
	\end{itemize}

\end{remark}


Let us then  consider nonlinear terms $b\in\mathbb A(\R;C^{\frac{\alpha}{2},\alpha}([0,T]\times\overline{\Omega}))$  of the form
\begin{align}\label{sep}
	b(t,x,\mu)=b_1(t,x)\s h(t,b_2(t,x)\mu)+b_0(t,x),\quad (t,x,\mu)\in[0,T]\times\overline{\Omega}\times\R.
\end{align}
We start by considering nonlinear terms of the form \eqref{sep} with $b_2\equiv1$. 

\begin{Thm}\label{c4} Let the conditions of Theorem \ref{t1} be fulfilled and assume that, for $j=1,2$, there exists $h_j\in \mathbb A(\R;C^{\frac{\alpha}{2}}([0,T]))$ such that
	\begin{align}\label{c4a}
		b_j(t,x,\mu)=b_{j,1}(t,x)h_j(t,\mu)+b_{j,0}(t,x),\quad (t,x,\mu)\in[0,T]\times\overline{\Omega}\times\R.
	\end{align}
	Assume also that,  for all $t\in(0,T)$, there exist $\mu_t\in\R$ and $n_t\in\mathbb N$ such that
	\begin{align}\label{c4b}
		\partial_\mu^{n_t} h_1(t,\ccdot)\not\equiv0 \text{ and } \partial_\mu^{n_t} h_1(t,\mu_t)=0,\quad t\in(0,T).
	\end{align}
	Moreover, we assume that 
	\begin{align}\label{c4c}
		b_{1,1}(t,x)\neq0,\quad  (t,x)\in Q,
	\end{align}
	and that for all $t\in(0,T)$ there exists $x_t\in\partial\Omega$ such that
	\begin{align}\label{c4d}
		b_{1,1}(t,x_t)=b_{2,1}(t,x_t)\neq0,\quad  t\in(0,T).
	\end{align}
	Then,  the condition \eqref{t1c} implies that $b_1=b_2$. Moreover, assuming that the conditions of Theorem \ref{t5} are fulfilled, the condition \eqref{t5a} implies that $b_1=b_2$.
\end{Thm}
Under stronger assumption imposed on the expression $h$ we can also consider  nonlinear terms of the form \eqref{sep} with $b_2\not\equiv1$.
\begin{Thm}\label{c5} Let the conditions of Theorem \ref{t1} be fulfilled and assume that, for $j=1,2$, there exists $h_j\in \mathbb A(\R;C^{\frac{\alpha}{2}}([0,T]))$ such that
	\begin{align}\label{c5a}
		b_j(t,x,\mu)=b_{j,1}(t,x)h_j(t,b_{j,2}(t,x)\mu)+b_{j,0}(t,x),\quad (t,x,\mu)\in[0,T]\times\overline{\Omega}\times\R.
	\end{align}
	Assume also that,  for all $t\in(0,T)$,  there exists $n_t\in\mathbb N$ such that
	\begin{align}\label{c5b}
		\partial_\mu^{n_t} h_1(t,\ccdot)\not\equiv0 \text{ and } \partial_\mu^{n_t} h_1(t,0)=0,\quad t\in(0,T).\
	\end{align}
	Moreover, we assume that 
	\begin{align}\label{c5c}
		b_{1,1}(t,x)\neq0 \text{ and } b_{1,2}(t,x)\neq0,\quad  (t,x)\in Q,
	\end{align}
	and that for all $t\in(0,T)$ there exists $x_t\in\partial\Omega$ such that
	\begin{align}\label{c5d}
		b_{1,1}(t,x_t)=b_{2,1}(t,x_t)\neq0 \text{ and }  b_{1,2}(t,x_t)=b_{2,2}(t,x_t)\neq0\quad  t\in(0,T).
	\end{align}
	Then,  the condition \eqref{t1c} implies that $b_1=b_2$. Moreover, assuming that the conditions of Theorem \ref{t5} are fulfilled, the condition \eqref{t5a} implies that $b_1=b_2$.
\end{Thm}

Finally,  we consider nonlinear terms $b\in\mathbb A(\R;C^{\frac{\alpha}{2},\alpha}([0,T]\times\overline{\Omega}))$ satisfying
\begin{align}\label{sep2}
	b(t,x,\mu)=b_1(t,x)G(x,b_2(t,x)\mu)+b_0(t,x),\quad (t,x,\mu)\in[0,T]\times\Omega\times\R.
\end{align}

\begin{corollary}\label{c7} Let the conditions of Theorem \ref{t1} be fulfilled and assume that, for $j=1,2$, there exists  $G\in \mathbb A(\R;C^{\alpha}(\overline{\Omega}))$ such that
	\begin{align}\label{c7a}
		b_j(t,x,\mu)=b_{j,1}(t,x)G(x,b_{j,2}(t,x)\mu)+b_{j,0}(t,x),\quad (t,x,\mu)\in[0,T]\times\Omega\times\R.
	\end{align}
	Assume also that one of the following condition is fulfilled: 
	\begin{itemize}
		\item[(i)] We have $b_{1,2}=b_{2,2}$, and  for all $x\in\Omega$,  there exists $n_x\in\mathbb N$ and $\mu_x\in\R$ such that
		\begin{align}\label{c7b}
			\partial_\mu^{n_x} G(\ccdot,\mu_x)\not\equiv0 \text{ and }  \partial_\mu^{n_x} G(x,\mu_x)=0,\quad x\in\Omega.
		\end{align}
		
		\item[(ii)] For all $x\in\Omega$,  there exists $n_x\in\mathbb N$  such that
		\begin{align}\label{c7c}
			\partial_\mu^{n_x} G(\ccdot,\mu_x)\not\equiv0\text{ and }  \partial_\mu^{n_x} G(x,0)=0,\quad x\in\Omega.
		\end{align}
	\end{itemize} 
	Moreover, we assume that condition \eqref{c5c} is fulfilled.
	Then,  the condition \eqref{t1c} implies that $b_1=b_2$. In addition, by assuming that the conditions of Theorem \ref{t5} are fulfilled, condition \eqref{t5a} implies that $b_1=b_2$.
\end{corollary}
\begin{remark}\label{r1}
	Let us observe that the results of Theorems \ref{c4} and \ref{c5} and Corollary \ref{c7} are mostly based on the conditions \eqref{c4b} and \eqref{c5c} imposed to  nonlinear terms of the form \eqref{sep}, and on the conditions \eqref{c7b} and \eqref{c7c} imposed to  nonlinear terms of the form \eqref{sep2}. These conditions are rather general and they  will be fulfilled in various situations for different class of functions. For instance, assuming that the function $h_1$ takes the form  
	$$h_1(t,\mu)=P(t,\mu)\exp(Q(t,\mu)),\quad (t,\mu)\in[0,T]\times\R$$
	with $P,Q\in \mathbb A(\R;C^{\frac{\alpha}{2}}([0,T]))$,  condition \eqref{c4b} will be fulfilled if we assume that there exists $\sigma\in C^{\alpha/2}([0,T])$ such that
	$$\left. \left[\partial_\mu P(t,\mu)+P(t,\mu)\partial_\mu Q(t,\mu)\right]\right|_{\mu=\sigma(t)}=0,\quad t\in[0,T].$$
	Such a condition will of course be fulfilled when $h_1(t,\mu)=\mu e^\mu$, $(t,\mu)\in[0,T]\times\R$.
	
	More generally, let $\sigma\in C^{\frac{\alpha}{2}}([0,T])$ be arbitrary chosen, and for each $t\in[0,T]$, consider $N_t\in\mathbb N$. Assuming that  the function $h_1$ satisfies the following property
	\begin{equation}\label{bb1}h_1(t,\mu)=\sum_{k=0}^{N_t-1}a_k(t)(\mu-\sigma(t))^k+ \underset{\mu\to\sigma(t)}{\mathcal O}\left((\mu-\sigma(t))^{N_t+1}\right),\quad t\in[0,T],\end{equation}
	with $a_0,\ldots,a_{N-1}\in C^{\frac{\alpha}{2}}([0,T])$ arbitrary chosen, one can easily check that condition \eqref{c4b} will be fulfilled since we have
	$$\partial_\mu^{N_t} h_1(t,\sigma(t))=0,\quad t\in[0,T].$$
	Condition \eqref{c5c} will be fulfilled under the same condition provided that $h_1$ satisfies \eqref{bb1} with $\sigma\equiv0$. Moreover, let $g\in C^{\alpha}(\overline{\Omega})$, and for each $x\in\Omega$, consider $N_x\in\mathbb N$. Assuming that  the function $G$ in \eqref{c7b} satisfies the following property
	\begin{equation}\label{bb2}G(x,\mu)=\sum_{k=0}^{N_x-1}\beta_k(x)(\mu-g(x))^k+ \underset{\mu\to g(x)}{\mathcal O}\left((\mu-g(x))^{N_x+1}\right),\quad x\in \Omega,\end{equation}
	with  the arbitrary chosen functions  $\beta_0,\ldots,\beta_{N-1}\in C^{\alpha}(\overline{\Omega})$, it is clear that condition \eqref{c7b} will be fulfilled since we have
	$$\partial_\mu^{N_x} G(x,g(x))=0,\quad x\in \Omega.$$ The same is true for  condition \eqref{c7c} when $G$ satisfies \eqref{bb2} with $g\equiv0$.
\end{remark}

Finally, via previous observations, we can also determine the order coefficients for \emph{linear} parabolic equations. 

\begin{corollary}[Global uniqueness with partial data]\label{cor: unique_poten_partial}
	Adopting all notations in Theorem \ref{t5}, let $q_j=q_j(t,x)\in C^\infty([0,T]\times \overline{\Omega})$ and $b_j(t,x,\mu)=q_j(t,x)\mu$ for $j=1,2$. Then \eqref{t1c} implies $q_1=q_2$ in $Q$.
\end{corollary}
We mention that we could also prove that the assumptions of Theorem \ref{t1} and $b_j(t,x,\mu)=q_j(t,x)\mu$, $j=1,2$, implies $q_1=q_2$. 

\subsection{Comments about our results}

To the best of our knowledge, Theorems \ref{t1} and \ref{t5} give the first positive answer to the inverse problem (IP1) for semilinear parabolic equations. In addition, the results of Theorem \ref{t1} and \ref{t5} extend the analysis of \cite{liimatainen2022uniqueness} that considered a problem similar to (IP1) for elliptic equations,  but which did not fully answered to the question raised by (IP1). In that sense, Theorems \ref{t1} and \ref{t5} give the first positive answer to  (IP1) for a class of elliptic PDEs as well. While Theorem \ref{t1} is stated for general class of parabolic equations, Theorem \ref{t5} gives a result with measurement restricted to a neighborhood of the back set with respect to a source $x_0\in\R^n\setminus\overline{\Omega}$ in the spirit of the most precise  partial data results stated for linear elliptic equations such as  \cite{KSU}. Note that in contrast to \cite{KSU} the source $x_0$ is not necessary outside the convex hull of $\overline{\Omega}$ and, as observed in \cite{KSU}, when $\Omega$ is convex the measurements of Theorem \ref{t5} can be restricted to any open set of $\partial\Omega$. Even for linear equations, Theorems \ref{t1} and \ref{t5} improve in precision and generality the earlier works of \cite{CK,Is0} dealing with determination of time dependent coefficients appearing in linear parabolic equations. 

We gave a positive answer to the problem (IP2) and show that the gauge breaks for three different classes of semilinear terms: 
\begin{itemize}
	\item[1)] Semilinear terms with prescribed information in Corollaries \ref{c1} and \ref{c2},
	
	\item[2)] Polynomial semilinear terms in Theorem \ref{c3},
	
	\item[3)] Semilinear terms with separated variables of the form \eqref{sep} or \eqref{sep2} in Theorems \ref{c4} and \ref{c5} and in Corollary \ref{c7}.
\end{itemize}  
This seems to be the most complete overview of situations where one can give a positive answer to problem (IP2). While \cite{liimatainen2022uniqueness} considered also such phenomenon for polynomial nonlinear terms and some specific examples, the conditions of Corollary \ref{c1} and \ref{c2}, Theorems \ref{c4} and \ref{c5} and Corollary \ref{c7}, leading to a positive answer for (IP2), seems to be  new. In Remark \ref{r1} we gave several concrete and general examples of semilinear terms satisfying the conditions of Theorems \ref{c4}, \ref{c5} and Corollary \ref{c7}.

The proof of our results are based on a combination of the higher order linearization technique, application of suitable class of geometric optics solutions for parabolic equations, Carleman estimates, properties of holomorphic functions and different properties of parabolic equations. Theorem \ref{t1} is deduced from the linearized result of Proposition \ref{p2} that we prove by using geometric optics solutions for parabolic equations. These solutions are built by using  the energy estimate  approach introduced in the recent work of \cite{Fe}. This allows us to consider problem (IP1) for general class of semilinear parabolic equations with variable coefficients. In Theorem \ref{t5}, we combine this class of geometric optics solutions with a  Carleman estimate with boundary terms stated in Lemma \ref{l7} in order  to restrict the boundary measurements to a part of the boundary. Note that the weight under consideration in Lemma \ref{l7} is not a limiting Carleman weight for parabolic equations. This is one reason why we can not apply such Carleman estimates for making also a restriction on the support of the Dirichlet input in Theorem \ref{t5}.

It is worth mentioning that our results for problem (IP1) and (IP2) can be applied to \emph{inverse source problems} for nonlinear parabolic equations. This important application is discussed in Section \ref{sec_simul}. There we show how the nonlinear interaction allows to solve this problem for general classes of source terms, depending simultaneously on the time and space variables. Corresponding problem problem for linear equations can not be solved uniquely (see Example \ref{rmk: example} or  e.g. \cite[Appendix A]{KSXY}). In that sense our analysis exhibits a new consequences of the nonlinear interaction, already considered for example in \cite{FO20,KLU,LLLS,LLLS2019partial, LLST2022inverse,KU,KU2019partial,FLL2021inverse}), by showing how  nonlinearity can help for the resolution of inverse source problems for parabolic equations.

We remark that while Theorem \ref{t1} is true for $n\geq2$,
we can only prove Theorem \ref{t5} for dimension $n\geq3$. The fact that we can not prove Theorem \ref{t5} for $n=2$ is related to the Carleman estimate of Lemma \ref{l7} that we can only derive for $n\geq3$. Since this Carleman estimate is a key ingredient in the proof of Theorem \ref{t5}, we need to exclude the case $n=2$ in the statement of this result.

Finally, let us observed that, under the suitable assumption of simplicity stated in Assumption \ref{assump simple}, the result of Theorem \ref{t1} can be applied to the determination of a semilinear term for reaction diffusion equations on a Riemannian manifold with boundary equipped with a time-dependent metric.

\subsection{Outline of the paper}
This article is organized as follows. In Section \ref{sec_prel}, we consider the forward problem by proving the well-posedness of \eqref{eq1} under suitable condition, and we recall some properties of the higher order linearization method for parabolic equations. Section \ref{sec_proof_thm} is devoted to the proof of Theorem \ref{t1} while in Section \ref{sec_smoothgo} we prove Proposition \ref{p2}. In Section \ref{sec_proof thm 1.2} we prove Theorem \ref{t5}, and in Section \ref{sec_gau_breaking} we consider our results related to problem (IP2). Finally, in Section \ref{sec_simul} we discuss the applications of our results to inverse source problems for parabolic equations. In the Appendix \ref{sec_app} the proof of the Carleman estimates of Lemma \ref{l7}.

\section{The forward problem and higher order linearization}\label{sec_prel}
Recall that in this article we assume that there is a solution $u_0$ to \eqref{eq1} corresponding to a lateral boundary data $f_0$.

\subsection{Well-posedness for Dirichlet data close to $f_0$}

In this subsection, we consider the well-posedness for the problem \eqref{eq1}, whenever the boundary datum $f$ is sufficiently close to $f_0$ with respect to $C^{1+\frac{\alpha}{2},2+\alpha}([0,T]\times \partial\Omega)$. For this purpose, we consider the Banach space $\mathcal{K}_0$ with the norm of the space $C^{1+\frac{\alpha}{2},2+\alpha}([0,T]\times \partial\Omega)$. Our local well-posedness result is stated as follows.

\begin{prop}\label{p1} Let $a:=(a_{ik})_{1 \leq i,k \leq n} \in C^\infty([0,T]\times\overline{\Omega};\R^{n\times n})$ satisfy \eqref{ell},  
	$\rho \in C^\infty([0,T]\times\overline{\Omega};\R_+)$ and $b\in C^\infty(\R;C^{\frac{\alpha}{2},\alpha}([0,T]\times\overline{\Omega}))$. We assume also that there exists a boundary value $f_0\in \mathcal K_0$ such that the problem \eqref{eq1} with $f=f_0$  admits a unique solution  $u_0\in C^{1+\frac{\alpha}{2},2+\alpha}([0,T]\times\overline{\Omega}))$. Then  there exists $\epsilon>0$ depending on  $a$, $\rho$, $b$, $f_0$, $\Omega$, $T$, such that, for all $f\in \mathbb B(f_0,\epsilon)$, the problem \eqref{eq1} admits a unique solution $u_f\in C^{1+\frac{\alpha}{2},2+\alpha}([0,T]\times\overline{\Omega}))$ satisfying
	\begin{align}\label{p1a}
		\norm{u_f-u_0}_{C^{1+\frac{\alpha}{2},2+\alpha}([0,T]\times\overline{\Omega}))}\leq C\norm{f-f_0}_{ C^{1+\frac{\alpha}{2},2+\alpha}([0,T]\times\partial\Omega)}.
	\end{align}
	Moreover, the map $\mathbb B(f_0,\epsilon)\ni f\mapsto u_f\in C^{1+\frac{\alpha}{2},2+\alpha}([0,T]\times\overline{\Omega}))$ is $C^\infty$ in the Fr\'echet sense.
	
\end{prop}

\begin{proof} 
	Let us first observe that we may look for a solution $u_f$ by splitting it into two terms by $u_f=u_0+v$, where $v$ solves
	\begin{align}\label{eq2}
		\begin{cases}
			\rho\s \partial_t v+\mathcal A(t)  v+  b(t,x,v+u_0)-b(t,x,u_0)=0  & \mbox{in}\ (0,T)\times\Omega ,
			\\
			v=h &\mbox{on}\ (0,T)\times\partial\Omega,\\
			v(0,x)=0 &\text{for } x\in\Omega,
		\end{cases}
	\end{align}
	with $h:=f-f_0$.
	Therefore, it is enough for our purpose to show that there exists $\epsilon>0$ depending on $a$, $\rho$, $b$, $f_0$, $\Omega$, $T$, such that for $h\in \mathbb B(0,\epsilon)$ the problem \eqref{eq2} admits a unique solution $v_h\in C^{1+\frac{\alpha}{2},2+\alpha}([0,T]\times\overline{\Omega})$ satisfying
	\begin{align}\label{p1b}
		\norm{v_h}_{C^{1+\frac{\alpha}{2},2+\alpha}([0,T]\times\overline{\Omega})}\leq C\norm{h}_{C^{1+\frac{\alpha}{2},2+\alpha}([0,T]\times\partial\Omega)}.
	\end{align}
	We introduce the spaces
	\begin{align*}
		\begin{split}
			\mathcal H_0:=&\left\{u\in C^{1+\frac{\alpha}{2},2+\alpha}([0,T]\times\overline{\Omega}):\ u|_{\{0\}\times\overline{\Omega}}\equiv 0,\ \left. \partial_tu\right|_{\{0\}\times\partial\Omega}\equiv 0\right\}, \\
			\mathcal L_0:=&\left\{F\in C^{\frac{\alpha}{2},\alpha}([0,T]\times\overline{\Omega}):\ F|_{\{0\}\times\partial\Omega}\equiv 0\right\}.
		\end{split}
	\end{align*}
	Then, let us introduce the map $\mathcal G$ defined by 
	\begin{align*}
		\mathcal G:\mathcal K_0\times\mathcal H_0 &\to \mathcal L_0\times\mathcal K_0 ,\\
		(h,v) &\mapsto\left(\rho \s\partial_t v+\mathcal A(t)  v+ b(t,x,v+u_0)-b(t,x,u_0), \, v|_{\Sigma}-h\right).
	\end{align*}
	
	We will find a solution to \eqref{eq1} by applying the implicit function theorem to the map $\mathcal G$. Using the fact that $b\in C^\infty(\R;C^{\frac{\alpha}{2},\alpha}([0,T]\times\overline{\Omega}))$, it follows that the map $\mathcal G$ is $C^\infty$ on $\mathcal K_0\times\mathcal H_0$ in the Fr\'echet sense. Moreover, we have $\mathcal G(0,0)=(0,0)$ and
	$$\partial_v\mathcal G(0,0)w=\left(\rho\s \partial_t w+\mathcal A(t)  w+ \partial_\mu b(t,x,u_0)w, \,  w|_{\Sigma}\right).$$
	In order to apply the implicit function theorem, we will prove that the map $\partial_v\mathcal G(0,0)$ is an isomorphism from $\mathcal H_0$ to $\mathcal L_0\times\mathcal K_0$. For this purpose, let us fix $(F,h)\in \mathcal L_0\times\mathcal K_0$ and let us consider the linear parabolic problem
	\begin{align}\label{eq101}
		\begin{cases}
			\rho \partial_t w+\mathcal A(t)  w+ \partial_\mu b(t,x,u_0)w=F(t,x)  & \mbox{in}\ Q ,
			\\
			w=h &\mbox{on}\ \Sigma,\\
			w(0,x)=0 &\ x\in\Omega.
		\end{cases}
	\end{align}
	Applying \cite[Theorem 5.2, Chapter IV, page 320]{LSU}, 
	we deduce that problem \eqref{eq101} admits a unique solution $w\in \mathcal H_0$ satisfying
	$$\norm{w}_{C^{1+\frac{\alpha}{2},2+\alpha}([0,T]\times\overline{\Omega})}\leq C\left(\norm{F}_{C^{\frac{\alpha}{2},\alpha}([0,T]\times\overline{\Omega})}+\norm{h}_{C^{1+\frac{\alpha}{2},2+\alpha}([0,T]\times\partial\Omega)}\right),$$
	for some constant $C>0$ independent of $w$, $F$ and $h$.
	From this result we deduce that $\partial_v\mathcal G(0,0)$ is an isomorphism from $\mathcal H_0$ to $\mathcal L_0\times\mathcal K_0$.

	Therefore, applying the implicit function theorem (see e.g. \cite[Theorem 10.6]{renardy2006introduction}), we deduce that there exists $\epsilon>0$ depending on  $a$, $b$, $\rho$, $f_0$, $\Omega$, $T$, and a smooth map $\psi$ from $\mathbb B(0,\epsilon)$ to $\mathcal H_0$, such that, for all $h\in \mathbb B(0,\epsilon)$, we have $\mathcal G(h,\psi(h))=(0,0)$.
	This proves that $v=\psi(h)$ is a solution of \eqref{eq2} for all $h\in \mathbb B(0,\epsilon)$. 
	
	For the uniqueness of the solution of \eqref{eq2}, let us consider $v_1 \in C^{1+\frac{\alpha}{2},2+\alpha}([0,T]\times\overline{\Omega})$ to be a solution of \eqref{eq2}, and let us show that $v_1=v$. For this purpose, we fix $w=v_1-v$ and notice that $w$ solves the following initial boundary value problem
	\begin{align}\label{eq22}
		\begin{cases}
			\rho\s \partial_t w+\mathcal A(t)  w+  q(t,x)w=0  & \mbox{in}\ Q ,
			\\
			w=0 &\mbox{on}\ \Sigma,\\
			w(0,x)=0 &\text{for } x\in\Omega,
		\end{cases}
	\end{align}
	with
	$$q(t,x)=\int_0^1\partial_\mu b(t,x,sv_1(t,x)+(1-s)v(t,x)+u_0(t,x))\, ds,\quad (t,x)\in Q.$$
	Then the uniqueness of the solutions of \eqref{eq22} implies that $w\equiv0$, and by the same way that $v=v_1$. Therefore, $v=\psi(h)$ is the unique solution of \eqref{eq2}.  Combining this with the fact that $\psi$ is smooth from $\mathbb B(0,\epsilon)$ to $\mathcal H_0$ and $\psi(0)=0$, we obtain \eqref{p1a}.  Finally, recalling that, for all $f\in \mathbb B(f_0,\epsilon)$, $u_f=u_0+\psi(f-f_0)$ with $\psi$ a $C^\infty$ map from $\mathbb B(0,\epsilon)$ to $C^{1+\frac{\alpha}{2},2+\alpha}([0,T]\times\overline{\Omega}))$, we deduce that the map $\mathbb B(f_0,\epsilon)\ni f\mapsto u_f\in C^{1+\frac{\alpha}{2},2+\alpha}([0,T]\times\overline{\Omega}))$ is $C^\infty$.
	This completes the proof of the proposition.
\end{proof}

\subsection{Linearizations of the problem}
In this subsection we assume that the conditions of Proposition \ref{p1} are fulfilled. Let us introduce  $m\in\mathbb N\cup\{0\}$ and consider the parameter $s=(s_1,\ldots,s_{m+1})\in(-1,1)^{m+1}$. 
Fixing $h_1,\ldots,h_{m+1}\in \mathbb B(0,\frac{\epsilon}{m+1})$,
we consider $u=u_{s}$ the solution of
\begin{align}\label{eq4}
	\begin{cases}
		\rho(t,x) \partial_t u(t,x)+\mathcal A(t)  u(t,x)+ b(t,x,u(t,x))=0  & \mbox{in}\ Q ,
		\\
		u=f_0+\displaystyle\sum_{i=1}^{m+1}s_ih_i &\mbox{on}\ \Sigma,\\
		u(0,x)=0 &\text{for } x\in\Omega.
	\end{cases}
\end{align}
Following the proof of Proposition \ref{p1}, we know that the map $s\mapsto u_s$ is lying in 
$$C^\infty\left((-1,1)^{m+1};C^{1+\frac{\alpha}{2},2+\alpha}([0,T]\times\overline{\Omega})\right),$$
then we are able to differentiate \eqref{eq4} with respect to the $s$ parameter.

Let us introduce
the solution of the first linearized problem
\begin{align}\label{eq5}
	\begin{cases}
		\rho(t,x) \partial_t v+\mathcal A(t)  v+ \partial_\mu b(t,x,u_0)v=0  & \mbox{in}\ Q ,
		\\
		v=h &\mbox{on}\ \Sigma,\\
		v(0,x)=0 &\text{for } x\in\Omega.
	\end{cases}
\end{align}
Using the facts that $u_s|_{s=0}=u_0$ and that the map $s\mapsto u_s$ is smooth, we see that that if $v_{\ell}$ is the solution of \eqref{eq5} with    $h=h_\ell$, $\ell=1,\ldots,m+1$,  then we have
\begin{align}\label{l2a}
	\left.\partial_{s_\ell}u_{s}\right|_{s=0}= v_{\ell},\quad \ell=1,\ldots,m+1,
\end{align}
in the sense of functions taking values in $C^{1+\frac{\alpha}{2},2+\alpha}([0,T]\times\overline{\Omega})$.

Now let us turn to the expression $\partial_{s_{\ell_1}}\partial_{s_{\ell_2}}u_{s}\big|_{s=0}$, $\ell_1,\ell_2=1,\ldots,m+1$.
For this purpose, we introduce the function
$w_{\ell_1,\ell_2}\in C^{1+\frac{\alpha}{2},2+\alpha}([0,T]\times\overline{\Omega})$ solving the second linearized problem
\begin{align}\label{eq6}
	\begin{cases}
		\rho(t,x) \partial_t w_{\ell_1,\ell_2}+\mathcal A(t)  w_{\ell_1,\ell_2}+ \partial_\mu b(t,x,u_0) w_{\ell_1,\ell_2}=-\partial_\mu^2 b(t,x,u_0)v_{\ell_1}v_{\ell_2}  & \mbox{in}\ Q,
		\\
		w_{\ell_1,\ell_2}=0 &\mbox{on}\ \Sigma,\\
		w_{\ell_1,\ell_2}(0,x)=0 &\text{for } x\in\Omega.
	\end{cases}
\end{align}
Repeating the above arguments, we obtain that 
\begin{align}\label{l3a}
	\left. \partial_{s_{\ell_1}}\partial_{s_{\ell_2}}u_{s}\right|_{s=0}= w_{\ell_1,\ell_2}
\end{align}
is the solution to \eqref{eq6}.

Then, one can prove by iteration the following result.

\begin{lem}[Higher order linearizations]\label{l5} Let $m\in\mathbb N$. The function 
	\begin{align}\label{l5a}
		w^{(m+1)}=\left.\partial_{s_1}\partial_{s_2}\cdots\partial_{s_{m+1}}u_{s}\right|_{s=0}
	\end{align}
	is well defined in the sense of functions taking values in $C^{1+\frac{\alpha}{2},2+\alpha}([0,T]\times\overline{\Omega})$. Moreover, $w^{(m+1)}$ solves the $(m+1)$-th order linearized problem 
	\begin{align}\label{eq7}
		\begin{cases}
			\rho(t,x) \partial_t w^{(m+1)}+\mathcal A(t)  w^{(m+1)}+ \partial_\mu b(t,x,u_0) w^{(m+1)}=H^{(m+1)}  & \mbox{in}\ Q,
			\\
			w^{(m+1)}=0 &\mbox{on}\ \Sigma,\\
			w^{(m+1)}(0,x)=0 &\mbox{for } x\in\Omega.
		\end{cases}
	\end{align}
	Here,  we have 
	\begin{align}\label{Hm}
		H^{(m+1)}=-\partial_\mu^{m+1} b(t,x,u_0)v_{1}\cdots v_{m+1} +K^{(m+1)},
	\end{align}
	where all the functions are evaluated at the point $(t,x)$ and $K^{(m+1)}(t,x)$ depends only on $a$, $\rho$, $\Omega$, $T$, 
	$\partial_\mu^k b(t,x,u_0)$,  $k=0,\ldots,m$, $v_1,\ldots, v_{m+1}$, and $w^{(k+1)}$, for $k=1,\ldots, m-1$, 
	which are respectively the solution of \eqref{eq4} with    $h=h_\ell$, $\ell=1,\ldots,m+1$.
\end{lem}

\section{Proof of Theorem~\ref{t1}}
\label{sec_proof_thm}

In this section we will prove Theorem~\ref{t1} by admitting the proof of the following denseness result.

\begin{prop}\label{p2} Adopting all the conditions of Theorem \ref{t1}, consider $F\in C([0,T]\times\overline{\Omega})$ and $m\in\mathbb N$. Assume that the following identity 
	\begin{align}\label{p2a}
		\int_0^T\int_\Omega Fv_1\cdots v_m\s w \, dxdt=0
	\end{align}
	holds true for all $v_j\in C^{1+\frac{\alpha}{2},2+\alpha}([0,T]\times\overline{\Omega}))$ and $w\in C^{1+\frac{\alpha}{2},2+\alpha}([0,T]\times\overline{\Omega}))$ solving the following equations
	\begin{align}\label{p2b}
		\begin{cases}
			\rho(t,x) \partial_t v_j+\mathcal A(t)  v_j+ \partial_\mu b(t,x,u_0)v_j=0,  & \mbox{in}\ Q ,\\
			v_j(0,x)=0 &\mbox{for } x\in\Omega,
		\end{cases}
	\end{align}
	for $j=1,\ldots, m$, and 
	\begin{align}\label{p2c}
		\begin{cases}
			-\partial_t(\rho(t,x) w)+\mathcal A(t)  w+ \partial_\mu b(t,x,u_0)w=0,  & \mbox{in}\ Q ,
			\\
			w(T,x)=0 &\mbox{for } x\in\Omega,
		\end{cases}
	\end{align}
	respectively. Then $F\equiv 0$.
	
\end{prop}

We next use the above proposition to show Theorem \ref{t1}. The proof of the proposition will be postponed to Section \ref{sec_smoothgo}.

\begin{proof}[Proof of Theorem \ref{t1}]
	We will prove this theorem in two steps. We will start by proving that the assumption
	\begin{align*}
		\mathcal N_{b_1}(f)=\mathcal N_{b_2}(f), \text{ for all }f\in \mathbb B(f_0,\epsilon), 
	\end{align*}
	implies that 
	\begin{align}\label{t1h}
		\partial_\mu^k b_1\left(t,x,u_{1,0}(t,x)\right)=\partial_\mu^k b_2\left(t,x,u_{2,0}(t,x)\right),\quad (t,x)\in [0,T]\times\overline{\Omega},
	\end{align}
	holds true for all $k\in\mathbb N$, with $u_{j,0}$ being the solution of \eqref{eq1} with $b=b_j$ and $f=f_0$, for $j=1,2$. Then, we will complete the proof by showing that \eqref{t1h} implies the claim
	\begin{align*}
		b_1=S_\varphi b_2.
	\end{align*}
	for some $\varphi\in   C^{1+\frac{\alpha}{2},2+\alpha}([0,T]\times\overline{\Omega})$ satisfying the conditions \eqref{gauge1}.
	
	\smallskip
	
	\textbf{Step 1. Determination of the Taylor coefficients} 
	
	\smallskip
	
	\noindent We will show \eqref{t1h} holds true, for all $k\in\mathbb N$, by recursion. For $j=1,2$, consider $v_{j,1}\in C^{1+\frac{\alpha}{2},2+\alpha}([0,T]\times\overline{\Omega}))$ satisfying \eqref{p2b} with $b=b_j$ and  $w\in C^{1+\frac{\alpha}{2},2+\alpha}([0,T]\times\overline{\Omega}))$ satisfying \eqref{p2b} with $b=b_1$. We assume here that
	$\left. v_{1,1}\right|_{\Sigma}=h=\left. v_{2,1}\right|_{\Sigma}$ for some $h\in\mathbb B(0,1)$. Applying the first order linearization we find 
	$$\left. \partial_{\nu} v_{1,1}\right|_{\Sigma}=\partial_s\mathcal N_{b_1}(f_0+sh)=\partial_s\mathcal N_{b_2}(f_0+sh)=\left.\partial_{\nu} v_{2,1}\right|_{\Sigma}.$$
	Thus, fixing $v_1=v_{1,1}-v_{2,1}$, we deduce that $v_1$ satisfies the following conditions
	\begin{align}\label{first line equ}
		\begin{cases}
			\rho(t,x) \partial_t v_1+\mathcal A(t)  v_1+ \partial_\mu b_1(t,x,u_0)v_1 =G(t,x)  & \mbox{in}\ Q ,
			\\
			v_1=\partial_{\nu_a}v_1=0 &\mbox{on}\ \Sigma,\\
			v(0,x)=0 &\mbox{for } x\in\Omega,
		\end{cases}
	\end{align}
	where 
	$$G(t,x)=\left(\partial_\mu b_2(t,x,u_{2,0}(t,x))-\partial_\mu b_1(t,x,u_{1,0}(t,x))\right)v_{2,1}(t,x),\quad (t,x)\in Q.$$
	Multiplying the equation \eqref{first line equ} by $w$ and integrating by parts, we obtain the identity 
	\begin{align*}
		\begin{split}
			&\int_0^T\int_{\Omega}\left(\partial_\mu b_2(t,x,u_{2,0})-\partial_\mu b_1(t,x,u_{1,0})\right)v_{2,1}w \, dx dt \\
			&= \int_{0}^{T}\int_{\Omega} \left( \rho(t,x) \partial_t v_1+\mathcal A(t)  v_1+ \partial_\mu b_1(t,x,u_0)v_1\right) w\, dxdt  \\
			&=\underbrace{\int_{0}^{T}\int_{\Omega} \left( -\partial_t (\rho w)  +\mathcal{A}(t)w+ \partial_\mu b_1(t,x,u_0)w \right) v_1 \, dxdt}_{\text{Since } 
				v_1=\partial_{\nu(a)} v_1 =0 \text{ on }\Sigma, \ v_1(0,x)=0 \text{ and } w(T,x)=0 \text{ in } \Omega.} \\
			&= 0.
		\end{split}
	\end{align*}
	Using the the fact that  $v_{2,1}$ can be seen as an arbitrary chosen element of $C^{1+\frac{\alpha}{2},2+\alpha}([0,T]\times\overline{\Omega}))$ satisfying \eqref{p2b} and applying Proposition \ref{p2}, we deduce that \eqref{t1h} holds true for $k=1$. Moreover, by the unique solvability of \eqref{eq5} 
	we deduce that $v_{1,1}=v_{2,1}$ in $Q$. 
	
	Now, let us fix $m\in\mathbb N$ and assume that, for $k=1,\ldots,m$,  \eqref{t1h} holds true and
	\begin{equation}\label{eq:sols_agree}
		w_1^{(k)}=w_2^{(k)} \text{ in } Q.
	\end{equation} 
	Consider $v_j\in C^{1+\frac{\alpha}{2},2+\alpha}([0,T]\times\overline{\Omega}))$ satisfying \eqref{p2b} with $b=b_1$ for $j=1,\ldots, m+1$, and  $w\in C^{1+\frac{\alpha}{2},2+\alpha}([0,T]\times\overline{\Omega}))$ satisfying \eqref{p2b} with $b=b_1$. We fix $h_j=\left. v_j\right|_{\Sigma}$, $j=1,\ldots,m+1$, and proceeding to the higher order linearization described in Lemma \ref{l5}, we obtain
	$$
	\left.\partial_{\nu(a)} w^{(m+1)}_j\right|_{\Sigma}=\left.\partial_{s_{1}}\ldots\partial_{s_{m+1}}\mathcal N_{b_j}(f_0+s_1h_1+\ldots+s_{m+1}h_{m+1})\right|_{s=0},$$
	with $s=(s_1,\ldots,s_{m+1})$ and $w^{(m+1)}_j$ solving \eqref{eq7} 
	as $b=b_j$, for $j=1,2$. Then the condition \eqref{t1c} implies 
	$$\left.\partial_{\nu(a)} w^{(m+1)}_1\right|_{\Sigma}=\left.\partial_{\nu(a)} w^{(m+1)}_2\right|_{\Sigma}.$$
	Fixing $w^{(m+1)}=w^{(m+1)}_1-w^{(m+1)}_2$ and applying Lemma \ref{l5}, we deduce that $w^{(m+1)}$ satisfies the following condition
	\begin{align}\label{equ K}
		\begin{cases}
			\rho(t,x) \partial_t w^{(m+1)}+\mathcal A(t)  w^{(m+1)}+ \partial_\mu b_1(t,x,u_0) w^{(m+1)}=\mathcal K  & \mbox{in}\ Q,
			\\
			w^{(m+1)}=\partial_{\nu_a} w^{(m+1)}=0 &\mbox{on}\ \Sigma,\\
			w^{(m+1)}(0,x)=0 &\mbox{for } x\in\Omega,
		\end{cases}
	\end{align}
	where $\mathcal K=\left(\partial_\mu^{m+1} b_2(t,x,u_{2,0})-\partial_\mu^{m+1} b_1(t,x,u_{1,0})\right)v_1\cdots v_{m+1}$. Here we used the assumption for this recursion argument that \eqref{t1h} and \eqref{eq:sols_agree} hold true for $k=1,\ldots,m$. Multiplying the equation \eqref{equ K} by $w$ and integrating by parts, we obtain
	$$\int_0^T\int_\Omega \left(\partial_\mu^{m+1} b_2(t,x,u_{2,0})-\partial_\mu^{m+1} b_1(t,x,u_{1,0})\right)v_1\cdots v_{m+1}\s w\, dxdt=0.$$
	Applying again Proposition \ref{p2}, we find that \eqref{t1h} holds true for $k=1,\ldots,m+1$. By unique solvability of \eqref{eq5}, we also have $w_1^{(m+1)}=w_2^{(m+1)}$ in $Q$. It follows that \eqref{t1h} holds true for all $k\in\mathbb N$.
	
	\smallskip
	
	\textbf{Step 2. Gauge invariance.} 
	
	\smallskip
	
	\noindent In this step we will show that \eqref{t1d} holds with some $\varphi\in   C^{1+\frac{\alpha}{2},2+\alpha}([0,T]\times\overline{\Omega})$ satisfying \eqref{gauge1}. We will choose here $\varphi=u_{2,0}-u_{1,0}$ and, thanks to \eqref{t1c}, we know that $\varphi$ fulfills condition \eqref{gauge1}. We fix $(t,x)\in[0,T]\times\overline{\Omega}$ and  consider the map
	$$G_j:\R\ni\mu\mapsto b_j(t,x,u_{j,0}(t,x)+\mu)-b_j(t,x,u_{j,0}(t,x)),\quad j=1,2.$$
	For $j=1,2$, using the fact that $b_j\in\mathbb A(\R;C^{\frac{\alpha}{2},\alpha}([0,T]\times\overline{\Omega}))$, we deduce that the map  $G=G_1-G_2$ is analytic with respect to $\mu\in\R$. It is clear that 
	$$
	G(0)=G_1(0)-G_2(0)=0-0=0.
	$$
	Moreover, \eqref{t1h} implies that
	$$G^{(k)}(0)=\partial_\mu^kb_1(t,x,u_{1,0}(t,x))-\partial_\mu^kb_2(t,x,u_{2,0}(t,x))=0,\quad k\in\mathbb N.$$
	Combining this with the fact that $G$ is analytic with respect to $\mu\in\R$, we deduce that there must exist $\delta>0$ such that
	$$G(\mu)=0,\quad \mu\in(-\delta,\delta).$$
	Then, the unique continuation of analytic functions implies that $G\equiv0$. It follows that, for all $\mu\in\R$, we have
	$$b_1(t,x,u_{1,0}(t,x)+\mu)-b_1(t,x,u_{1,0}(t,x))=b_2(t,x,u_{2,0}(t,x)+\mu)-b_2(t,x,u_{2,0}(t,x)).$$
	Recalling that 
	$$-b_j(t,x,u_{j,0}(t,x))=\rho(t,x) \partial_t u_{j,0}(t,x)+\mathcal A(t)u_{j,0}(t,x),\quad j=1,2,$$
	we deduce that
	\begin{align}\label{t1i}
		\begin{split}
			&b_1(t,x,u_{1,0}(t,x)+\mu)\\
			=&b_2(t,x,u_{2,0}(t,x)+\mu)+b_1(t,x,u_{1,0}(t,x))-b_2(t,x,u_{2,0}(t,x))\\
			=&b_2(t,x,u_{2,0}(t,x)+\mu)+\rho(t,x) \partial_t(u_{2,0}-u_{1,0})(t,x)+\mathcal A(t)(u_{2,0}-u_{1,0})(t,x)\\
			=&b_2(t,x,u_{2,0}(t,x)+\mu)+\rho(t,x) \partial_t \varphi(t,x)+\mathcal A(t)\varphi(t,x),\quad \mu\in\R.
		\end{split}
	\end{align}
	Considering \eqref{t1i} with $\mu_1=u_{1,0}(t,x)+\mu$, we obtain
	$$\begin{aligned}b_1(t,x,\mu_1)&=b_2(t,x,u_{2,0}(t,x)-u_{1,0}(t,x)+\mu_1)+\rho(t,x) \partial_t \varphi(t,x)+\mathcal A(t)\varphi(t,x)\\
		&=b_2(t,x,\varphi(t,x)+\mu_1)+\rho(t,x) \partial_t \varphi(t,x)+\mathcal A(t)\varphi(t,x),\quad \mu_1\in\R.\end{aligned}$$ 
	Using the fact that here $(t,x)\in[0,T]\times\overline{\Omega}$ is arbitrary chosen we deduce that \eqref{t1d} holds true with  $\varphi=u_{2,0}-u_{1,0}$. This completes the proof of the theorem.
\end{proof}


\section{Proof of Proposition \ref{p2}}
\label{sec_smoothgo}

In order to prove Proposition \ref{p2}, we need to construct special solutions, which helps us to prove the completeness of products of solutions.

\subsection{Constructions of geometric optics solutions}

For the proof of Proposition~\ref{p2} will need to consider the construction of \emph{geometrical optics} (GO in short) solutions. More precisely, fixing $q\in L^\infty(Q)$ we will consider GO solutions to the equation
\begin{align}\label{eqGO1sm}
	\begin{cases}
		\rho(t,x)\partial_tv+\mathcal A(t) v+qv =0  & \mbox{in}\ Q ,
		\\
		v(0,x)=0 &\mbox{for } x\in\Omega,
	\end{cases}
\end{align}
as well as GO solutions for the formal adjoint equation 
\begin{align}\label{eqGO2sm}
	\begin{cases}
		-\rho(t,x)\partial_tw+\mathcal A(t) w-\partial_t\rho(t,x)w+qw=0  & \mbox{in}\ Q ,
		\\
		w(T,x)=0 &\mbox{for } x\in\Omega,
	\end{cases}
\end{align}
belonging to the  space  $H^1(0,T;L^2(\Omega))\cap L^2(0,T;H^2(\Omega))$.
Following the recent construction of \cite{Fe}, in this section we give a  construction of these GO solutions that will depend on a large asymptotic positive parameter $\tau$ with $\tau\gg1$ and  concentrate on  geodesics with respect to the metric $g(t)=\rho(t,\ccdot)a(t,\ccdot)^{-1}$ in $\overline{\Omega_1}$ that passes through a point $x_0\in\partial\Omega_1$, whenever Assumption \ref{assump simple} holds. Here $\Omega_1$ is a domain satisfying Assumption \ref{assump simple} containing  $\overline{\Omega}$. Let us construct the GO solution as follows.

First, we  consider solutions of the form
\begin{align}\label{GO7}
	v(t,x)=e^{\tau^2t+\tau \psi(t,x)}\left[c_{+}(t,x)+R_{+,\tau}(t,x)\right],\quad (t,x)\in Q,
\end{align}
and 
\begin{align}\label{GO8}
	w(t,x)=e^{-\tau^2t-\tau \psi(t,x)}\left[c_{-}(t,x)+R_{-,\tau}(t,x)\right],\quad (t,x)\in Q,
\end{align}
to equations \eqref{eqGO1sm} and \eqref{eqGO2sm} respectively. The phase functions and principal terms $c_\pm$ of the GOs will be constructed  by using polar normal coordinate on the manifold $\left(\overline{\Omega_1}, g(t)\right)$, for $t\in[0,T]$.

Let us define the differential operators $L_\pm$, $P_{\tau,\pm}$ on $\Omega_1$  by
\begin{align}\label{l}
	\begin{split}
		L_{+}=&\rho(t,x)\partial_t+\mathcal A(t)+q(t,x),\\
		L_{-}=&-\rho(t,x)\partial_t+\mathcal A(t) -\partial_t\rho(t,x)+q(t,x),
	\end{split}
\end{align}
and 
\begin{align}\label{P_pm}
	P_{\tau,\pm}= e^{\mp (\tau^2t+\tau \psi(t,x))}L_{\pm}\left( e^{\pm (\tau^2t+\tau \psi(t,x))} \right).
\end{align}
Via a straightforward computation, we can write 
$$P_{\tau,\pm}v=\tau^2\, \mathcal I v+ \tau\, \mathcal J_{\pm}v + L_{\pm}v,$$
where  
\begin{align}\label{I}
	\mathcal I=\rho(t,x)-\sum_{i,k=1}^na_{ik}(t,x)\partial_{x_i}\psi\partial_{x_k}\psi,
\end{align}
and 
\begin{align}\label{J}
	\mathcal J_\pm v:=\mp2\sum_{i,k=1}^na_{ik}(t,x)\partial_{x_i}\psi\partial_{x_k} v+\left[\rho(t,x)\partial_t \psi\pm\mathcal A(t)\psi\right]v.
\end{align}

Next we want to choose $\psi$ in such a way that the eikonal equation $\mathcal I=0$ is satisfied in $Q$. Hence, after choosing $\psi$, we seek for $c_{\pm}$ solving the transport equations
\begin{align}\label{trans1}
	\mathcal J_+c_{+}=0\quad \text{ and } \quad  \mathcal J_-c_{-}=0.
\end{align}
Since for all $t\in[0,T]$  the Riemannian manifold $\left(\overline{\Omega_1}, g(t)\right)$ is assumed to be simple, the eikonal equation $\mathcal I=0$ can be solved globally on $\overline{Q}$. This is known, but let us show how it is done.  
For this, let us fix $x_0\in\partial\Omega_1$ and consider the polar normal coordinates $(r,\theta)$ on $\left(\overline{\Omega_1}, g(t)\right)$ given by $x=\exp_{x_0}(r\s \theta)$, where $r>0$ and 
$$
\theta\in S_{x_0,t}(\overline{\Omega_1}):=\left\{v\in \R^n:\, |v|_{g(t)[x_0]}=1\right\}.
$$
According to the Gauss lemma, in these coordinates the metric takes the form 
$$
dr^2+g_0(t,r,\theta),
$$ 
where $g_0(t,r,\theta)$ is a metric defined on $S_{x_0,t}(\overline{\Omega_1})$, which depends smoothly on $t$ and $r$. 
In fact, we choose 
\begin{align}
	\label{psi1}
	\psi(t,x)=\textrm{dist}_{g(t)}(x_0,x), \quad (t,x) \in [0,T]\times\overline{\Omega},
\end{align}
where $\text{dist}_{g(t)}$ is the Riemannian distance function associated with the metric $g(t)$, $t\in[0,T]$. 
As $\psi$ is given by $r$ in the polar normal coordinates,  one can easily check that $\psi$ solves $\mathcal I=0$. 

\
Let us now turn to the transport equations \eqref{trans1}. 
We write $c_{\pm}(t,r,\theta)=c_{\pm}(t,\exp_{x_0}(r\s \theta))$
and use this notation to indicate the representation in the polar normal coordinates also for other functions. 
Then, using this notation and following \cite[Section 5.1.2]{Fe}, we deduce that, in these polar normal coordinates with respect to $x_0\in\partial \Omega_1$, the equations in \eqref{trans1} become 
\begin{equation}\label{eq:transport_polar}
	\pd_rc_{\pm}+\left({\pd_r\beta \over 4\beta}\right)c_{\pm}\mp \frac{\partial_t\psi(t,r,\theta)}{2}c_{\pm}=0
\end{equation}
with $\beta(t,r,\theta)=\det \left(g_0(t,r,\theta)\right)$. Note that in this equation  there is  no differentiation in the $\theta$-variable. This fact will allow use to localized GO solutions near geodesics. We fix
$$r_0=\inf_{t\in[0,T]}\textrm{dist}_{g(t)}(\partial\Omega_1,\overline{\Omega})$$
and recall that $r_0>0$.
For any $h\in C^\infty( S_{x_0,t}(\overline{\Omega}_1))$ and $\chi\in C^\infty_0(0,T)$, 
the functions 
\begin{align}\label{b12}
	&c_{+}(t,r,\theta)=\chi(t)h(\theta)\beta(t,r,\theta)^{-1/4}\exp\left(  \int_{r_0}^{r}\frac{\partial_t\psi(t,s,\theta)}{2} \, ds\right),\\
	&c_{-}(t,r,\theta)=\chi(t)\beta(t,r,\theta)^{-1/4}\exp\left( - \int_{r_0}^{r}\frac{\partial_t\psi(t,s,\theta)}{2}\, ds\right),	
\end{align}
are respectively solutions of the transport equations \eqref{trans1}. Moreover, the regularity of the coefficients $\rho$, $a$ and the simplicity of the manifold $\left(\overline{\Omega_1}, g(t)\right)$ implies that solutions of the transport equation \eqref{eq:transport_polar} $c_\pm\in C^\infty([0,T]\times\overline{\Omega})$.

In order to complete our construction of GO solutions, we need to show that it is possible to construct the remainder terms 
$$R_{\pm,\tau}\in H^1(0,T; L^2(\Omega))\cap L^2(0,T;H^2(\Omega))$$ 
satisfying the decay property
\begin{align}\label{GO17}
	\left\| R_{\pm,\tau}\right\|_{L^2(Q)}\leq C\,|\tau|^{-1},
\end{align}
for some constant $C>0$ independent of $\tau$, positive and large enough, as well as the initial and final condition
$$
R_{+,\tau}(0,x)=R_{-,\tau}(T,x)=0,\quad x\in\Omega.$$
For this purpose, we recall that for $\psi$ given by \eqref{psi1} we have $P_{\tau,\pm}=L_{\pm}+\tau \mathcal J_\pm$ with $L_\pm$ and $\mathcal J_\pm$ defined by \eqref{l}--\eqref{J}. Then,  according to \eqref{trans1}, we have 
$$\begin{aligned}L_\pm\left[e^{\tau^2t+\tau \psi(t,x)}c_{\pm}(t,x)\right]
	=&e^{\tau^2t+\tau \psi(t,x)}P_{\tau,\pm}c_{\pm}(t,x)\\
	=&e^{\tau^2t+\tau \psi(t,x)}L_{\pm}c_{\pm}.\end{aligned}$$
Therefore, the conditions $L_{+}v=0$ and $L_-w=0$ are  fulfilled if and only if $R_{\pm,\tau}$
solves
$$P_{\tau,\pm}R_{\pm,\tau}(t,x)=-L_\pm c_{\pm}(t,x),\quad (t,x)\in(0,T)\times\Omega.$$
We will choose $R_{\pm,\tau}$ to be the solution of the following IBVP
\begin{align}\label{R+}
	\begin{cases}
		P_{\tau,+}R_{+,\tau}(t,x)  =  -L_+c_{+}(t,x) & \text{ in } 
		Q,\\
		R_{+,\tau}(t,x)  =  0  & \text{ on } \Sigma, \\  
		R_{+,\tau}(0,x)  =  0 &  \text{ for } x \in \Omega,
	\end{cases}
\end{align}
and
\begin{align}\label{R-}
	\begin{cases}
		P_{\tau,-}R_{-,\tau}(t,x)  =  -L_-c_{-}(t,x) &  \text{ in } 
		Q,\\
		R_{-,\tau}(t,x)  =  0 &   \text{ on } \Sigma, \\  
		R_{-,\tau}(T,x)  =  0 & \text{ for } x \in \Omega.
	\end{cases}
\end{align}
Note that $L_\pm c_\pm$ is independent of $\tau$. We give the following extension of the energy estimate approach under consideration \cite{Fe} for problem \eqref{R+}-\eqref{R-}.
\begin{prop}\label{p3} There exists $\tau_0>0$, depending only on $\Omega$, $T$, $a$, $\rho$, $q$, such that problem \eqref{R+}--\eqref{R-} admits a unique solution $R_{\pm,\tau}\in H^1(0,T;L^2(\Omega))\cap L^2(0,T;H^2(\Omega))$ satisfying the estimate 
	\begin{align}\label{p3b}
		\tau \left\| R_{\pm,\tau}\right\|_{L^2(Q)}+\tau^{\frac{1}{2}}\left\| R_{\pm,\tau}\right\|_{L^2(0,T;H^1(\Omega))}\leq C\left\| L_\pm c_{\pm}\right\|_{L^2(Q)},\quad \tau>\tau_0.
	\end{align}
	
\end{prop}
\begin{proof}
	The proof of this proposition is based on arguments similar to \cite[Proposition 4.1]{Fe} that we adapt to problem \eqref{R+}--\eqref{R-} whose equations are more general than the ones under consideration in \cite{Fe}. For this reason and for sake of completeness we give the full proof of this proposition. We only show the result for $R_{+,\tau}$, the same property for $R_{-,\tau}$ can be deduced by applying similar arguments. Let us first observe that from the classical theory of existence of solutions for linear PDEs one can check that \eqref{R+} admits a unique solution $R_{+,\tau}\in H^1(0,T;L^2(\Omega))\cap L^2(0,T;H^2(\Omega))$ and we only need to check estimate \eqref{p3b}.

	In all this proof, we set 
	$$
	v=R_{+,\tau}, \quad  K=-L_+c_{+},
	$$
	and without loss of generality we assume that both $v$ and $K$ are real valued.
	We fix $\lambda>0$ and we multiply \eqref{R+} by $ve^{\lambda \psi}$ in order to get
	\begin{align*}
		\int_Q \left(-L_+c_{+} \right) e^{\lambda \psi} \, dxdt =&\int_Q Kve^{\lambda \psi}\, dxdt \\
		=&\int_Q \left(\rho\partial_tv+\mathcal A(t)v+\tau\, \mathcal J_{+}v+qv\right)ve^{\lambda \psi}\, dxdt\\
		:=&I+II+III,
	\end{align*}
	where
	\begin{align*}
		I&=\int_Q \left(\rho\partial_tv+qv+\tau \rho(t,x)\partial_t \psi v\right)ve^{\lambda \psi}\, dxdt, \\
		II&=\int_Q (\mathcal A(t)v)ve^{\lambda \psi}\, dxdt, \\
		III&=\tau \int_Q\left(-2 \sum_{i,k=1}^na_{ij}(t,x)\partial_{x_i}\psi\partial_{x_k} v+\mathcal A(t)\psi v\right)ve^{\lambda \psi}\, dxdt.
	\end{align*}

	For $I$, using the fact that $v_{|t=0}=0$ and integrating by parts, we get
	\begin{align*}
		I=&\frac{1}{2}\int_\Omega \rho(T,x)v(T,x)^2e^{\lambda \psi(T,x)}\, dx-\frac{1}{2}\int_Q\partial_t\left(\rho e^{\lambda \psi}\right)v^2\, dxdt \\
		&+\int_Q \left(qv+\tau \rho(t,x)\partial_t \psi v\right)ve^{\lambda \psi}\, dxdt\\
		\geq& -\frac{1}{2}\int_Q\partial_t\left(\rho e^{\lambda \psi}\right)v^2\, dxdt+\int_Q \left(qv+\tau \rho(t,x)\partial_t \psi v\right)ve^{\lambda \psi}\, dxdt\\
		\geq &-\left(C_1+C_2\tau+C_3\lambda\right) \int_Qv^2e^{\lambda \psi}\, dxdt,
	\end{align*}
	with $C_1,C_2, C_3$ three positive constants independent of $\lambda$ and $\tau$.

	For $II$, using the fact that $v|_{\Sigma}=0$, applying \eqref{ell} and integrating by parts, we find
	\begin{align*}
		II&=\int_Q \left(\sum_{i,k=1}^na_{ik}(t,x)\partial_{x_i}v\partial_{x_k} v\right)e^{\lambda \psi}\, dxdt+\frac{1}{2}\int_Q \sum_{i,k=1}^na_{ik}(t,x)\partial_{x_i}(v^2)\partial_{x_k} \left(e^{\lambda \psi}\right)dxdt\\
		&\geq c\int_Q |\nabla_x v|^2e^{\lambda \psi}\, dxdt-\frac{1}{2}\int_Q v^2\mathcal A(t)\left(e^{\lambda \psi}\right)dxdt\\
		&\geq c\int_Q |\nabla_x v|^2e^{\lambda \psi}\, dxdt-\left(C_4\lambda+C_5\lambda^2\right) \int_Qv^2e^{\lambda \psi}\, dxdt,
	\end{align*}
	with constants $c, C_4,C_5>0$  independent of $\lambda$ and $\tau$.

	Finally, for $III$, using the that $\mathcal I=0$, with $\mathcal I$ defined by \eqref{I}, we find
	\begin{align*}
		III=&-\tau \int_Q \sum_{i,k=1}^na_{ik}(t,x)\partial_{x_i}\psi\partial_{x_k} (v^2)e^{\lambda \psi}\, dxdt+\tau \int_Q(\mathcal A(t)\psi) v^2e^{\lambda \psi}\, dxdt\\
		=& \tau \int_Q \sum_{i,k=1}^na_{ik}(t,x)\partial_{x_i}\psi\partial_{x_k}\left(e^{\lambda \psi}\right) v^2\, dxdt\\
		=& \tau\lambda  \int_Q \left(\sum_{i,k=1}^na_{ik}(t,x)\partial_{x_i}\psi\partial_{x_k}\psi\right) v^2e^{\lambda \psi}\, dxdt\\
		=& \tau\lambda  \int_Q \rho(t,x) v^2e^{\lambda \psi}\, dxdt\\
		\geq &\left(\inf_{(t,x)\in Q}\rho(t,x)\right)\tau\lambda\int_Q  v^2e^{\lambda \psi}\, dxdt\\
		:=&C_6\tau\lambda\int_Q  v^2e^{\lambda \psi}\, dxdt.
	\end{align*}
	Combining these estimates of $I$, $II$ and $III$, we find
	\begin{align*}
		\int_Q Kve^{\lambda \psi}\, dxdt\geq &c\int_Q |\nabla_x v|^2e^{\lambda \psi}\, dxdt \\
		&+\left(-C_1-C_2\tau-C_3\lambda-C_4\lambda-C_5\lambda^2+C_6\tau\lambda\right)\int_Q  v^2e^{\lambda \psi}\,dxdt.
	\end{align*}
	Choosing $\lambda=\frac{3C_2}{C_6}$ and 
	$$\tau_0=\frac{3\left(\frac{C_1}{\lambda}+C_3+C_5\lambda\right)}{C_6},$$
	we deduce that
	$$\int_Q Kve^{\lambda \psi}\, dxdt\geq c\int_Q |\nabla_x v|^2e^{\lambda \psi}\, dxdt+\frac{C_6}{3}\tau\lambda\int_Q  v^2e^{\lambda \psi}\, dxdt,\quad \tau>\tau_0.$$
	Applying Cauchy-Schwarz inequality, for $\tau>\tau_0$, we get
	$$ c\int_Q |\nabla_x v|^2e^{\lambda \psi}\, dxdt+\frac{C_6}{3}\tau\lambda\int_Q  v^2e^{\lambda \psi}\, dxdt\leq \left(\tau^{-1}\int_Q K^2e^{\lambda \psi}\, dxdt\right)^{\frac{1}{2}}\left(\tau\int_Q v^2e^{\lambda \psi}\, dxdt\right)^{\frac{1}{2}},$$
	which implies that
	$$\tau^{\frac{1}{2}}\int_Q |\nabla_x v|^2\, dxdt+\tau\int_Q  v^2\, dxdt\leq C\int_Q K^2\, dxdt,\quad \tau>\tau_0,$$
	where $C>0$ is a constant independent of $\tau$. From this last estimate, we deduce \eqref{p3b}.
\end{proof}

Note that the the energy estimate  \eqref{p3b} is only subjected to the requirement that $\psi$ solves the eikonal equation $\mathcal I=0$ in $Q$. For this result the simplicity assumption is not required.

Applying Proposition \ref{p3}, we deduce the existence of $R_{\pm,\tau}$ fulfilling  condition \eqref{R+}--\eqref{R-} and the decay estimate \eqref{GO17}.  Armed with these class of GO solutions we are now in position to complete the proof of Proposition \ref{p2}.

\subsection{Completion of the proof of Proposition \ref{p2}}

We will show Proposition \ref{p2} by iteration. 
\begin{proof}[Proof of Proposition \ref{p2}]
	We start by showing that \eqref{p2a} holds true for $m=1$. Note that by density, we can extend  \eqref{p2a} to $v_1\in H^1(0,T; L^2(\Omega))\cap L^2(0,T;H^2(\Omega))$ and $w\in H^1(0,T; L^2(\Omega))\cap L^2(0,T;H^2(\Omega))$ satisfying \eqref{p2b}--\eqref{p2c}. We fix $x_0\in\partial\Omega_1$, $t_0\in(0,T)$ and for $\chi_*\in C^\infty_0((-1,1)$ satisfying 
	$$\int_\R \chi_*(t)^2\, dt=1,$$
	and we set 
	\begin{align*}
		\chi_\delta (t) =\delta^{-\frac{1}{2}}\chi_*\left(\delta^{-1}(t-t_0)\right), \quad  \text{ for }\delta\in \left(0,\min(T-t_0,t_0)\right).
	\end{align*} 
	We consider $v_1$ (resp. $w$) of the form  \eqref{GO7} (resp. \eqref{GO8}) satisfying \eqref{eqGO1sm} (resp. \eqref{eqGO2sm}) with $q(t,x)=\partial_\mu b_1(t,x,u_{1,0}(t,x))$ (resp. $q(t,x)=\partial_\mu b_2(t,x,u_{2,0}(t,x))$)  and $R_{+,\tau}$ (resp. $R_{-,\tau}$) satisfying the decay property \eqref{GO17}. Here we choose $\chi=\chi_\delta$ in the expression of the function $v_1$ and $w$.
	Then, condition \eqref{p2a} implies that 
	\begin{align}\label{p2d}
		\lim_{\tau\to+\infty}\int_0^T\int_\Omega F(t,x)v_1w \, dx dt=0=\int_0^T\int_\Omega F(t,x)c_{+,0}c_{-,0}\, dx dt.
	\end{align}
	From now on, for  $t\in[0,T]$, we denote by 
	$\pd _+S_t(\overline{\Omega_1})$ the unit sphere bundle 
	$$
	\pd _+S_t(\overline{\Omega_1}):=\left\{(x,\theta)\in S_t(\overline{\Omega_1}):\, x\in\pd \Omega_1,\ \left\langle \theta,\nu_t(x)\right\rangle_{g(t)}<0\right\},$$
	where $\nu_t$ denotes the outward unit normal vector of $\partial\Omega_1$ with respect to the metric $g(t)$. We also denote for any $(y,\theta)\in \pd _+S_t(\overline{\Omega_1})$ by $\ell_{t,+}(y,\theta)$ the time of existence in $\overline{\Omega_1}$ of the maximal geodesic $\gamma_{y,\theta}$, with respect to the metric $g(t)$, satisfying $\gamma_{y,\theta}(0)=y$ and $\gamma_{y,\theta}'(0)=\theta$.

	Consider $\tilde{F}\in L^\infty((0,T)\times\Omega_1)$ defined by 
	\begin{align*}
		\tilde{F}(t,x)=\begin{cases}
			(\det(g(t))^{-\frac{1}{2}}F(t,x), &\text{ for }(t,x)\in Q\\
			0,  & \text{ for }(t,x)\in (0,T)\times (\Omega_1\setminus\Omega)
		\end{cases},
	\end{align*}
	then we have
	$$\int_0^T\int_{\overline{\Omega_1}} \tilde{F}(t,x)c_{+,0}c_{-,0}\, dV_{g(t)}(x)  dt=\int_0^T\int_{\Omega_1} \tilde{F}(t,x)c_{+,0}c_{-,0}\sqrt{\det(g(t))}\, dx dt=0,$$
	where $dV_{g(t)}$ is the Riemannian volume of $\left(\overline{\Omega_1},g(t)\right)$. Passing to polar normal coordinate, we obtain
	$$\int_0^T\chi_\delta(t)^2\int_0^{\ell_{t,+}(y,\theta)}\int_{S_{x_0,t}(\overline{\Omega_1})}h(\theta)\tilde{F}(t,r,\theta)\, drd\theta dt=0.$$
	Using the fact that $F\in C([0,T]\times\overline{\Omega})$, we deduce that $\tilde{F}\in C([0,T];L^\infty(\Omega_1))$ and sending $\delta\to0$, we obtain
	$$\int_0^{\ell_{t_0,+}(x_0,\theta)}\int_{S_{x_0,t_0}(\overline{\Omega_1})}h(\theta)\tilde{F}(t_0,r,\theta)\, drd\theta=0.$$
	Applying the fact that in this identity $h\in C^\infty( S_{x_0,t_0}(\overline{\Omega_1}))$ is arbitrary chosen, we deduce that
	$$\int_0^{\ell_{t_0,+}(x_0,\theta)}\tilde{F}(t_0,\gamma_{x_0,\theta}(s))\, ds=\int_0^{\ell_{t_0,+}(x_0,\theta)}\tilde{F}(t_0,r,\theta)\, dr=0,\quad (x_0,\theta)\in \pd _+S_{t_0}(\overline{\Omega_1}).$$
	Combining this  with the facts that in this identity $x_0\in\partial\Omega_1$ was arbitrary chosen, the manifold $\left(\overline{\Omega_1},g(t_0)\right)$ is assumed to be simple and  that the geodesic ray transform is injective 
	on simple manifolds, we deduce that $\tilde{F}(t_0,\ccdot)\equiv0$ on $\Omega_1$. Thus $F(t_0,\ccdot)\equiv0$. Combining this with the fact that here $t_0\in(0,T)$ is arbitrary chosen and $F\in C([0,T]\times\overline{\Omega})$, we deduce that $F\equiv 0$.
	
	Now let us fix $m\geq1$, and assume that \eqref{p2a} for this $m$  implies that $F\equiv0$. Fix $G\in C([0,T]\times\overline{\Omega})$  and assume that 
	$$\int_0^T\int_\Omega Gv_1\cdots v_{m+1}\s w\, dxdt=0$$
	for all $v_1,\ldots, v_{m+1}\in C^{1+\frac{\alpha}{2},2+\alpha}([0,T]\times\overline{\Omega}))$ satisfying \eqref{p2b} and all $w\in C^{1+\frac{\alpha}{2},2+\alpha}([0,T]\times\overline{\Omega}))$ satisfying \eqref{p2c}. 
	Fixing $F=Gv_1$, we deduce that $F\equiv 0$ and multiplying $F$ by an arbitrary chosen $w\in C^{1+\frac{\alpha}{2},2+\alpha}([0,T]\times\overline{\Omega}))$ satisfying \eqref{p2c}, and integrating, we deduce that
	$$\int_0^T\int_\Omega Gv_1 w\, dxdt=0$$
	for all $v_1\in C^{1+\frac{\alpha}{2},2+\alpha}([0,T]\times\overline{\Omega}))$ satisfying \eqref{p2b} and all $w\in C^{1+\frac{\alpha}{2},2+\alpha}([0,T]\times\overline{\Omega}))$ satisfying \eqref{p2c}. Then, the above argumentation implies that $G\equiv0$. This proves the assertion.
\end{proof}

\section{Proof of Theorem \ref{t5}}\label{sec_proof thm 1.2}
In all this section, we assume that $n\geq3$ and $\psi(x)=|x-x_0|$, $x\in\Omega$,  for $x_0\in\R^n\setminus\overline{\Omega}$. Note that the function $\psi$ satisfies the eikonal equation 
\begin{align}\label{eikonal psi}
	\left|\nabla_x \psi(x)\right|=1, \text{ for }x\in \Omega.
\end{align}
We start by considering the following new Carleman estimate whose proof is postponed to Appendix \ref{sec_app}.

\begin{lem}\label{l7} Let $q\in L^\infty(Q)$ and $v\in H^1(Q)\cap L^2(0,T;H^2(\Omega))$ satisfy the condition 
	\begin{equation}\label{l7a}v|_{\Sigma}=0,\quad v|_{t=0}=0.\end{equation}
	Then, there exists $\tau_0>0$ depending on $T$, $\Omega$ and $\norm{q}_{L^\infty(Q)}$ such that  for all $\tau>\tau_0$ the following estimate 
	\begin{align}\label{l7b} \begin{split}
			&\tau\int_0^T\int_{\Gamma_{+}(x_0)}e^{-2(\tau^2t+\tau\psi(x))}\abs{\partial_\nu v}^2\abs{\partial_\nu\psi(x) } d\sigma(x)dt+\tau^2\int_Qe^{-2(\tau^2t+\tau\psi(x))}\abs{v}^2dxdt\\
			\leq & C\left(\int_Qe^{-2(\tau^2t+\tau\psi(x))}\abs{(\partial_t-\Delta_x+q)v}^2\, dxdt\right. \\
			& \qquad \left. +\tau\int_0^T\int_{\Gamma_{-}(x_0)}e^{-2(\tau^2t++\tau\psi(x))}\abs{\partial_\nu v}^2\abs{\partial_\nu\psi(x) }d\sigma(x)\, dt\right)
	\end{split}\end{align}
	holds true.
\end{lem}

Armed with these results we are know in position to complete the proof of Theorem \ref{t5}.

\begin{proof}[Proof of Theorem \ref{t5}]
	Following the proof of Theorem \ref{t1}, we only need to prove that \eqref{t1h} holds true. We will prove this by a recursion argument. Let us first observe that since $\tilde{\Gamma}$ is a neighborhood of $\Gamma_-(x_0)$, there exists $\epsilon>0$ such that $B(x_0,\epsilon)\cap \overline{\Omega}=\emptyset$ and for all $y\in B(x_0,\epsilon)$, we have 
	$$
	\Gamma_-(y,\epsilon):=\left\{x\in\partial\Omega:\ (x-y)\cdot\nu(x)\leq\epsilon\right\}\subset \tilde{\Gamma}.
	$$
	We start by considering \eqref{t1h} for $k=1$. For $j=1,2$, consider $v_{j,1}\in C^{1+\frac{\alpha}{2},2+\alpha}([0,T]\times\overline{\Omega}))$ satisfying \eqref{p2b} with $b=b_j$ and  $w\in C^{1+\frac{\alpha}{2},2+\alpha}([0,T]\times\overline{\Omega}))$ satisfying \eqref{p2c} with $b=b_1$. We assume here that
	$v_{1,1}|_{\Sigma}=h=v_{2,1}|_{\Sigma}$ for some $h\in \mathcal K_0$.  Fixing $v_1=v_{1,1}-v_{2,1}$, we deduce that $v_1$ satisfies the following conditions 
	\begin{align*}
		\begin{cases}
			\partial_t v_1-\Delta  v_1+ \partial_\mu b_1(t,x,u_{1,0})v_1 =F(t,x)  & \mbox{in}\ Q ,
			\\
			v_1=\partial_{\nu}v_1=0 &\mbox{on}\ (0,T)\times \tilde \Gamma,\\
			v(0,x)=0 &\mbox{for } x\in\Omega,
		\end{cases}
	\end{align*}
	with
	$$F(t,x)=\left(\partial_\mu b_2(t,x,u_{2,0}(t,x))-\partial_\mu b_1(t,x,u_{1,0}(t,x))\right)v_{2,1}(t,x),\quad (t,x)\in Q.$$
	Multiplying the above equation by $w$ and integrating by parts, we obtain
	$$\int_0^T\int_{\Omega}\left(\partial_\mu b_2(t,x,u_{2,0})-\partial_\mu b_1(t,x,u_{1,0})\right)v_{2,1}w \,  dx dt-\int_{\Sigma}\partial_\nu v_1(t,x) w(t,x)\, d\sigma(x)dt=0.$$
	Moreover, applying the first order linearization we find
	\begin{align*}
		\left.\partial_{\nu} v_{1,1}\right|_{(0,T)\times \Gamma_-(y,\epsilon)}=\left.\partial_s\mathcal N_{b_1}(f_0+sh)\right|_{(0,T)\times \Gamma_-(y,\epsilon)}
		=&\left.\partial_s\mathcal N_{b_2}(f_0+sh)\right|_{(0,T)\times \Gamma_-(y,\epsilon)}\\
		=&\left.\partial_{\nu} v_{2,1}\right|_{(0,T)\times \Gamma_-(y,\epsilon)}.
	\end{align*}
	and it follows that 
	\begin{equation}\label{t5e}
		\begin{split}
			&\int_0^T\int_{\Omega}(\partial_\mu b_2(t,x,u_{2,0})-\partial_\mu b_1(t,x,u_{1,0}))v_{2,1}w \, dx dt\\
			=&\int_0^T\int_{\partial\Omega\setminus \Gamma_-(y,\epsilon)}\partial_\nu v_1(t,x) w(t,x) \, d\sigma(x) dt
		\end{split}
	\end{equation}		
	with  $v_{2,1}$ (resp. $w$) an arbitrary chosen element of $C^{1+\frac{\alpha}{2},2+\alpha}([0,T]\times\overline{\Omega}))$ satisfying \eqref{p2b} (resp. \eqref{p2c}).

	By density we can extend this identity to $v_{2,1}$ and $w$ two GO solutions of the form \eqref{GO7} and \eqref{GO8} with $\psi(x)=|x-y|$ and $c_\pm$ given by 
	\begin{align*}
		c_{+}(t,x)=&\chi_\delta(t)h\left(\frac{x-y}{|x-y|}\right)|x-y|^{-(n-1)/2},\\
		c_{-}(t,x)=&\chi_\delta(t)|x-y|^{-(n-1)/2},
	\end{align*}
	for $(t,x)\in[0,T]\times(\R^n\setminus\{y\})$ with $h\in C^\infty(\mathbb S^{n-1})$ and $$\chi_\delta (t) =\delta^{-\frac{1}{2}}\chi_*(\delta^{-1}(t-t_0)), \text{ for }\delta\in(0,\min(T-t_0,t_0)),$$ where $\chi_*\in C^\infty_0(0,T)$ satisfies
	$$\int_\R \chi_*(t)^2\, dt=1.$$
	Note that the construction of such GO solutions is a consequence of the fact that, we can find $\Omega_2$ an open neighborhood of $\overline{\Omega}$ such that $\psi\in C^\infty(\overline{\Omega_2})$  solves the eikonal equation
	$$\left|\nabla_x\psi(x)\right|^2=1,\quad x\in\Omega_2,$$
	as well as an application of Proposition \ref{p3}. In addition, we  built this class of GO solutions by following the arguments used in Section 5.1 where the polar normal coordinates will be replaced by polar coordinates centered at $y$. Note that in such coordinates $\psi=r$ and the transport equations \eqref{eq:transport_polar} are just 
	$$
	\partial_r c_\pm + \left(\frac{\partial_r\beta}{4\beta}\right) c_\pm=0,
	$$
	where $\beta$ is an angle dependent multiple of $r^{2(n-1)}$.

	With this choice of the functions $v_{2,1}$ and $w$, we obtain by Cauchy-Schwarz inequality that
	\begin{align}\label{t5f}
		\begin{split}
			&\abs{\int_0^T\int_{\partial\Omega\setminus \Gamma_-(y,\epsilon)}\partial_\nu v_1(t,x) w(t,x)\, d\sigma(x)dt} \\
			\leq & C\left(\int_0^T\int_{\partial\Omega\setminus \Gamma_-(y,\epsilon)}|\partial_\nu v_1(t,x)|^2e^{-2(\tau^2t+\tau\psi(x))} d\sigma(x)dt\right)^{\frac{1}{2}}.
		\end{split}
	\end{align}
	In addition, the Carleman estimate \eqref{l7b} and the fact that $\left. \partial_\nu v_1\right|_{(0,T)\times\Gamma_-(y)}=0$ imply that, for $\tau>0$ sufficiently large, we have 
	$$\begin{aligned}&\int_0^T\int_{\partial\Omega\setminus \Gamma_-(y,\epsilon)}|\partial_\nu v_1(t,x)|^2e^{-2(\tau^2t+\tau\psi(x))} \, d\sigma(x)dt\\
		\leq &C\epsilon^{-1}\int_0^T\int_{\partial\Omega\setminus \Gamma_-(y,\epsilon)}|\partial_\nu v_1(t,x)|^2e^{-2(\tau^2t+\tau\psi(x))}\partial_\nu\psi(x) \, d\sigma(x)dt\\
		\leq &C\epsilon^{-1}\int_0^T\int_{ \Gamma_+(y)}|\partial_\nu v_1(t,x)|^2e^{-2(\tau^2t+\tau\psi(x))}\partial_\nu\psi(x) \, d\sigma(x)dt\\
		\leq &\underbrace{C\tau^{-1}\int_Q \left|\partial_t v_1-\Delta  v_1+ \partial_\mu b_1(t,x,u_0)v_1\right|^2e^{-2(\tau^2t+\tau\psi(x))}\, dxdt}_{\text{Here we use the Carleman estimate \eqref{l7b} with }\left.\partial_\nu v_1\right|_{(0,T)\times\Gamma_-(y)}=0.}\\
		\leq &C\tau^{-1}\int_Q \left|(\partial_\mu b_2(t,x,u_0)-\partial_\mu b_1(t,x,u_0))v_{2,1}\right|^2e^{-2(\tau^2t+\tau\psi(x))}\, dxdt\\
		\leq &C\tau^{-1}\int_Q\left|\partial_\mu b_2(t,x,u_0)-\partial_\mu b_1(t,x,u_0)\right|^2|c_+|^2\, dxdt\\
		\leq &C\tau^{-1},\end{aligned}$$
	where $C>0$ is a constant independent of $\tau$. Therefore, for $\tau>0$ sufficiently large, we obtain 
	$$\abs{\int_0^T\int_{\partial\Omega\setminus \Gamma_-(y,\epsilon)}\partial_\nu v_1(t,x) w(t,x)\, d\sigma(x) dt}\leq C\tau^{-\frac{1}{2}}$$
	and in a similar way to Proposition \ref{p2}, sending $\tau\to+\infty$, we find 
	$$\int_Q(\partial_\mu b_2(t,x,u_0)-\partial_\mu b_1(t,x,u_0))c_+c_-\, dxdt=0.$$
	By using the polar coordinates, sending $\delta\to0$ and repeating the arguments of Proposition \ref{p2}, one can get
	$$\int_0^{+\infty}\int_{\mathbb S^{n-1}} G(t,y+r\theta)h(\theta)\, d\theta dr=0,\quad t\in(0,T),$$
	where $G:=\partial_\mu b_2(t,x,u_{2,0})-\partial_\mu b_1(t,x,u_{1,0})$ in $Q$  extended to  $(0,T)\times\R^n$ by zero.

	Using the fact that $h\in C^\infty(\mathbb S^{n-1})$ is arbitrary chosen, we get
	$$
	\int_0^{+\infty} G(t,y+r\theta)\, dr=0,\quad t\in(0,T),\ \theta\in\mathbb S^{n-1},
	$$
	and the condition on the support of $G$ implies that
	$$\int_\R G(t,y+s\theta) \, ds=0,\quad t\in(0,T),\ \theta\in\mathbb S^{n-1}.$$
	Since this last identity holds true for all $y\in B(x_0,\epsilon)$, we obtain
	\begin{equation}\label{t5g}\int_\R G(t,y+s\theta)\, ds=0,\quad t\in(0,T),\ \theta\in\mathbb S^{n-1},\ y\in B(x_0,\epsilon).\end{equation}
	In addition, since $G=0$ on 	$(0,T)\times\R^n\setminus Q$, we know that
	$$G(t,x)=0,\quad t\in(0,T),\ x\in B(x_0,\epsilon)$$
	and, combining this with \eqref{t5g}, we are in position to apply \cite[Theorem 1.2]{IM} in order to deduce that, for all $t\in(0,T)$, $G(t,\cdot)\equiv0$. It follows that  $G\equiv0$ and \eqref{t1h} holds true for $k=1$.
	
	Now, let us fix $m\in\mathbb N$ and assume that \eqref{t1h} holds true for $k=1,\ldots,m$. Consider $v_1,\ldots,v_{m+1}\in C^{1+\frac{\alpha}{2},2+\alpha}([0,T]\times\overline{\Omega}))$ satisfying \eqref{p2b} with $b=b_1$ and  $w\in C^{1+\frac{\alpha}{2},2+\alpha}([0,T]\times\overline{\Omega}))$ satisfying \eqref{p2b} with $b=b_1$. We fix $h_j=v_j|_{\Sigma}$, $j=1,\ldots,m+1$, and proceeding to the higher order linearization described in Lemma \ref{l5}, we obtain
	$$
	\left.\partial_{\nu(a)} w^{(m+1)}_j\right|_{\Sigma}=\left.\partial_{s_{1}}\ldots\partial_{s_{m+1}}\mathcal N_{b_j}(f_0+s_1h_1+\ldots+s_{m+1}h_{m+1})\right|_{s=0},$$
	with $s=(s_1,\ldots,s_{m+1})$ and $w^{(m+1)}_j$ solving \eqref{eq6} with $b=b_j$ for $j=1,2$. Then, \eqref{t5a} implies 
	$$\left.\partial_{\nu(a)} w^{(m+1)}_1\right|_{(0,T)\times \tilde{\Gamma}}=\left.\partial_{\nu(a)} w^{(m+1)}_2\right|_{(0,T)\times \tilde{\Gamma}}$$
	and, fixing $w^{(m+1)}=w^{(m+1)}_1-w^{(m+1)}_2$ in $Q$, and applying Lemma \ref{l5}, we deduce that $w^{(m+1)}$ satisfies the following condition
	\begin{align*}
		\begin{cases}
			\partial_t w^{(m+1)}-\Delta  w^{(m+1)}+ \partial_\mu b_1(t,x,u_0) w^{(m+1)}=\mathcal K  & \mbox{ in } Q,
			\\
			w^{(m+1)}=0 &\mbox{ on }\Sigma, \\
			\partial_{\nu} w^{(m+1)}=0 &\mbox{ on }(0,T)\times \tilde{\Gamma}, \\
			w^{(m+1)}(0,x)=0 &\mbox{ for } x\in\Omega,
		\end{cases}
	\end{align*}
	where $\mathcal K=\left(\partial_\mu^{m+1} b_2(t,x,u_{2,0})-\partial_\mu^{m+1} b_1(t,x,u_{1,0})\right)v_{1}\cdot\ldots\cdot v_{m+1}$. Multiplying this equation by $w$ and integrating by parts, we obtain
	$$\begin{aligned}&\int_0^T\int_\Omega (\partial_\mu^{m+1} b_2(t,x,u_{2,0})-\partial_\mu^{m+1} b_1(t,x,u_{1,0}))v_1\cdot\ldots\cdot v_{m+1}\cdot w\, dxdt\\
		&\quad -\int_0^T\int_{\partial\Omega\setminus \tilde{\Gamma}}\partial_\nu w^{(m+1)} w(t,x)d\sigma(x)\, dt=0.\end{aligned}$$
	We choose $v_{m+1}$ and $w$ two GO solutions of the form \eqref{GO7} and \eqref{GO8} with $\psi(x)=|x-y|$, $y\in B(x_0,\epsilon)$, and we fix 
	$$H(t,x)=\left(\partial_\mu^{m+1} b_2(t,x,u_{2,0})-\partial_\mu^{m+1} b_1(t,x,u_{1,0})\right)v_{1}\cdot\ldots\cdot v_{m}(t,x),\quad (t,x)\in Q$$
	that we extend by zero to $(0,T)\times\R^n$. Applying  Cauchy-Schwarz inequality and the fact that $\Gamma_-(y,\epsilon)\subset\tilde{\Gamma}$, we find
	\begin{align*}
		\begin{split}
			&\abs{\int_0^T\int_{\partial\Omega\setminus \tilde{\Gamma}}\partial_\nu w^{(m+1)}(t,x) w(t,x)\, d\sigma(x)dt} \\
			\leq & C\left(\int_0^T\int_{\partial\Omega\setminus \Gamma_-(y,\epsilon)}|\partial_\nu w^{(m+1)}(t,x)|^2e^{-2(\tau^2t+\tau\psi(x))} d\sigma(x)dt\right)^{\frac{1}{2}}.
		\end{split}
	\end{align*}
	and, applying the  Carleman estimate \eqref{l7b} and repeating the above argumentation, we have
	\begin{align*}
		\begin{split}
			&\abs{\int_0^T\int_{\partial\Omega\setminus \Gamma_-(y,\epsilon)}\partial_\nu w^{(m+1)}(t,x) w(t,x)\, d\sigma(x)dt}^2 \\
			\leq & C\tau^{-1}\int_Q|H(t,x)|^2|v_{m+1}(t,x)|^2e^{-2(\tau^2t+\tau\psi(x))} dxdt\\
			\leq & C\tau^{-1}.
		\end{split}
	\end{align*}
	Thus, we obtain 
	$$\lim_{\tau\to+\infty}\int_Q H(t,x) v_{m+1}(t,x) w(t,x)\, dxdt=0$$
	and repeating the above argumentation, we get
	$$\left(\partial_\mu^{m+1} b_2(t,x,u_{2,0})-\partial_\mu^{m+1} b_1(t,x,u_{1,0})\right)v_{1}\cdot\ldots\cdot v_{m}(t,x)=H(t,x)=0,\quad (t,x)\in Q.$$
	Multiplying this expression by any $w\in C^{1+\frac{\alpha}{2},2+\alpha}([0,T]\times\overline{\Omega}))$ satisfying \eqref{p2b} with $b=b_1$, we obtain 
	$$\int_Q \left(\partial_\mu^{m+1} b_2(t,x,u_{2,0})-\partial_\mu^{m+1} b_1(t,x,u_{1,0})\right)v_{1}\cdot\ldots\cdot v_{m}w\, dxdt=0$$
	and applying Proposition \ref{p2} we can conclude that \eqref{t1h} holds true for $k=1,\ldots,m+1$. It follows that \eqref{t1h} holds true for all $k\in\mathbb N$ and repeating the arguments used in the second step of the proof of Theorem \ref{t1} we can conclude that \eqref{t5a} implies \eqref{t1d} with the function $\varphi=u_{2,0}-u_{1,0}$ satisfying \eqref{t5b}-\eqref{t5bb}.	
\end{proof}

\section{Breaking the gauge class}\label{sec_gau_breaking}

This section is devoted to the proof of the positive answers that we  give to problem (IP2) in the theorems and corollaries of Section \ref{sec 2.2}. 

\begin{proof}[Proof of Corollary \ref{c1}]

	We start by assuming that the condition of Theorem \ref{t1} are fulfilled and by proving that \eqref{t1c} implies $b_1=b_2$. By Theorem \ref{t1}, condition \eqref{t1c} implies that there exists $\varphi\in   C^{1+\frac{\alpha}{2},2+\alpha}([0,T]\times\overline{\Omega})$ satisfying \eqref{gauge1} such that
	\begin{align}\label{c1b}
		b_1(t,x,\mu)=b_2(t,x,\mu+\varphi(t,x))+\rho(t,x) \partial_t \varphi(t,x)+\mathcal A(t)  \varphi(t,x),\quad (t,x,\mu)\in Q\times\R.
	\end{align}
	We will prove that $\varphi\equiv0$ which implies that $b_1=b_2$.
	Choosing $\mu=\kappa(t,x)$ and applying \eqref{c1a}, we obtain
	$$\begin{aligned}&\rho(t,x) \partial_t \varphi(t,x)+\mathcal A(t)  \varphi(t,x)+b_2(t,x,\kappa(t,x)+\varphi(t,x))-b_2(t,x,\kappa(t,x))\\
		=&b_1(t,x,\kappa(t,x))-b_2(t,x,\kappa(t,x))=0,\quad (t,x)\in Q.\end{aligned}$$
	Moreover, we have
	$$\begin{aligned}
		&b_2(t,x,\kappa(t,x)+\varphi(t,x))-b_2(t,x,\kappa(t,x))\\
		=&\left(\int_0^1\partial_\mu b_2(t,x,\kappa(t,x)+s\varphi(t,x))ds\right)\varphi(t,x)\\
		:=&q(t,x)\varphi(t,x),\quad \text{ for } (t,x)\in Q.\end{aligned}$$
	Therefore, $\varphi$ fulfills the following condition
	\begin{align}\label{c2ff}
		\begin{cases}
			\rho(t,x) \partial_t \varphi(t,x)+\mathcal A(t)  \varphi(t,x)+q(t,x)\varphi(t,x)  =  0 & \text{ in }
			Q,\\ 
			\varphi(t,x) =  0 &  \text{ on } \Sigma,\\
			\varphi(0,x)  =  0 & \text{ for } x \in \Omega,
		\end{cases}
	\end{align}
	and the uniqueness of solutions  for this problem implies that $\varphi\equiv0$. Thus, \eqref{c1b} implies $b_1=b_2$.

	Using similar arguments one can check that, by assuming  the conditions of Theorem \ref{t5},  \eqref{t5a} implies also that $b_1=b_2$.
\end{proof}

\begin{proof}[Proof of Corollary \ref{c2}]
	Let us  assume that the conditions of Theorem \ref{t1} and \eqref{t1c} are fulfilled. By Theorem \ref{t1} there exists $\varphi\in   C^{1+\frac{\alpha}{2},2+\alpha}([0,T]\times\overline{\Omega})$ satisfying \eqref{gauge1} such that 
	\eqref{c1b} is fulfilled. In particular, by choosing $\mu=0$, we have
	$$b_1(t,x,0)-b_2(t,x,0)=b_2(t,x,-\varphi(t,x))-b_2(t,x,0)+\rho(t,x) \partial_t \varphi(t,x)+\mathcal A(t)  \varphi(t,x)\text{ in }  Q.$$
	Combining this with  \eqref{c2b} and \eqref{gauge1}, we deduce that $\varphi$ fulfills the following condition
	\begin{align}\label{c2f}
		\begin{cases}
			\rho(t,x) \partial_t \varphi(t,x)+\mathcal A(t)  \varphi(t,x)+q(t,x)\varphi(t,x)  =  h(x)G(t,x) & \text{ in } 
			Q,\\ 
			\varphi(t,x)=\partial_{\nu(a)}\varphi(t,x)  =  0 & \text{ on }\Sigma,\\
			\varphi(0,x)  =  0 & \text{ for }x \in \Omega,
		\end{cases}
	\end{align}
	with
	$$q(t,x)=\int_0^1\partial_\mu b_2(t,x,s\varphi(t,x))\, ds,\quad (t,x)\in Q.$$
	Moreover, following the proof of Theorem \ref{t1}, we know that $\varphi=u_{2,0}-u_{1,0}$ and the additional assumption \eqref{c2c},
	\[
	u_{1,0}(\theta,x)=u_{2,0}(\theta,x),
	\]
	implies that
	$$\varphi(\theta,x)=0,\quad x\in\Omega.$$
	Combining this with \eqref{c2a}, 
	\[
	\inf_{x\in \Omega}\abs{G(\theta,x)}>0,
	\]
	\eqref{c2f} and applying  \cite[Theorem 3.4]{IY},  we obtain $f\equiv0$. Then the source term in \eqref{c2f} is zero and uniqueness of solutions implies that $\varphi\equiv0$. Thus, \eqref{c1b} implies $b_1=b_2$. The last statement of the corollary can be deduced from similar arguments.
\end{proof}

We are ready to prove Theorem \ref{c3}.
\begin{proof}[Proof of Theorem \ref{c3}]
	We assume first that  the conditions of Theorem \ref{t1} and \eqref{t1c} are fulfilled. Let us first prove that we can assume that $N_1=N_2$. Indeed, assuming that  $N_1\neq N_2$, we may assume without loss of generality that $N_1>N_2$. From \eqref{c1b} we deduce that 
	$$\begin{aligned}b_{1,N_1}(t,x)&=\lim_{\lambda\to+\infty}\frac{b_1(t,x,\lambda)}{\lambda^{N_1}}\\
		&=\lim_{\lambda\to+\infty}\frac{b_2(t,x,\lambda+\varphi(t,x))+\rho(t,x) \partial_t \varphi(t,x)+\mathcal A(t)  \varphi(t,x)}{\lambda^{N_1}}=0,\end{aligned}$$
	for $(t,x)\in Q$. In the same way, we can prove by iteration that 
	$$
	b_{1,N_1}=b_{1,N_1-1}=\cdots=b_{1,N_2+1}\equiv0.
	$$
	Therefore, from now on we assume that $N_1=N_2=N$. In view of \eqref{c1b}, for all $(t,x,\mu)\in Q\times\R$, we get by renumbering the sums 
	$$\begin{aligned}\sum_{k=0}^Nb_{1,k}(t,x)\mu^k
		=&\sum_{k=1}^Nb_{2,k}(t,x)\left(\sum_{j=1}^k\left(\begin{array}{l}k\\ j\end{array}\right)\varphi(t,x)^{k-j}\mu^j\right) \\
		&+\rho(t,x) \partial_t \varphi(t,x)+\mathcal A(t)  \varphi(t,x)+b_2(t,x,\varphi(t,x))\\
		=&\sum_{j=1}^N\left(\sum_{k=j}^Nb_{2,k}(t,x)\left(\begin{array}{l}k\\ j\end{array}\right)\varphi(t,x)^{k-j}\right)\mu^j \\
		&+\rho(t,x) \partial_t \varphi(t,x)+\mathcal A(t)  \varphi(t,x)+b_2(t,x,+\varphi(t,x)).\end{aligned}$$
	It follows that
	\begin{align}\label{c3h}
		b_{1,j}(t,x)=\sum_{k=j}^Nb_{2,k}(t,x)\left(\begin{array}{l}k\\ j\end{array}\right)\varphi(t,x)^{k-j},\quad (t,x)\in Q,\ j=1,\ldots,N,
	\end{align}
	and 
	\begin{align}\label{c3i}
		\begin{split}
			b_{1,0}(t,x)-b_{2,0}(t,x)
			=\rho(t,x) \partial_t \varphi(t,x)+\mathcal A(t)  \varphi(t,x)+b_2(t,x,\varphi(t,x))-b_2(t,x,0),
		\end{split}
	\end{align}
	for $(t,x)\in Q$.

	Applying \eqref{c3h} with $j=N$ and $j=N-1$, we obtain
	\begin{align}\label{c3j}
		b_{1,N}=b_{2,N},\quad b_{1,N-1}=b_{2,N}\varphi+b_{2,N-1}.
	\end{align}
	Moreover, the fact that the condition \eqref{c3b} holds true on the dense set $J$ combined with the fact that $b_{j,k}\in C([0,T]\times\overline{\Omega})$, $j=1,2$ and $k=N-1,N$, imply that
	\begin{align*}
		\min \left(\left|(b_{1,N-1}-b_{2,N-1})(t,x)\right|, \ \sum_{j=1}^2\left|(b_{j,N}-b_{j,N-1})(t,x)\right|\right)=0,\ (t,x)\in (0,T)\times\omega.
	\end{align*}
	This condition implies that for all $(t,x)\in(0,T)\times\omega$ we have either $b_{2,N-1}(t,x)=b_{1,N-1}(t,x)$ or $b_{j,N}(t,x)=b_{j,N-1}(t,x)$, $j=1,2$. Combining this with \eqref{c3j} and the assumption that $|b_{1,N}(t,x)|>0$ for $(t,x)\in J$,
	we deduce that $\varphi=0$ on $(0,T)\times\omega$. Thus \eqref{c3i} implies
	$$\left(b_{1,0}-b_{2,0}\right)(t,x)=\rho \partial_t \varphi(t,x)+\mathcal A(t)  \varphi(t,x)+b_2(t,x,\varphi(t,x))-b_2(t,x,0)=0,$$
	for $(t,x)\in (0,T)\times\omega$.
	Finally, the assumption \eqref{c3d} implies that $b_{1,0}=b_{2,0}$ everywhere on $Q$. Thus  $\varphi$ satisfies \eqref{c2ff} with
	$$q(t,x):=\int_0^1\partial_\mu b_2(t,x,s\varphi(t,x))\, ds,\quad (t,x)\in Q.$$
	Therefore, the uniqueness of solutions of \eqref{c2ff} implies that $\varphi\equiv0$ and it follows that $b_1=b_2$. The last statement of the theorem can be proved with similar arguments.
\end{proof}

\begin{proof}[Proof of Theorem \ref{c4}]
	We assume that the conditions of Theorem \ref{t1} and \eqref{t1c} are fulfilled.  We start by showing that $\partial_\mu h_1=\partial_\mu h_2$. For this purpose, combining \eqref{gauge1} and \eqref{c1b} with \eqref{c4a} we obtain
	$$b_{1,1}(t,x)h_1(t,\mu)+b_{1,0}(t,x)=b_{2,1}(t,x)h_2(t,\mu)+b_{2,0}(t,x)+\mathcal A(t)  \varphi(t,x),\quad (t,x,\mu)\in\Sigma\times\R.$$
	Differentiating both sides of this identity with respect to $\mu$, we get
	$$b_{1,1}(t,x)\partial_\mu h_1(t,\mu)=b_{2,1}(t,x)\partial_\mu h_2(t,\mu) ,\quad (t,x,\mu)\in\Sigma\times\R$$
	and \eqref{c4d} implies that
	$$b_{1,1}(t,x_t)(\partial_\mu h_1(t,\mu)-\partial_\mu h_2(t,\mu))=0 ,\quad (t,\mu)\in(0,T)\times\R,$$
	with $b_{1,1}(t,x_t)\neq0$, $t\in(0,T)$. It follows that $\partial_\mu h_1=\partial_\mu h_2$.
	
	Now let us show that the function $\varphi$ of \eqref{c1b} is identically zero. Fixing $t\in(0,T)$ and applying the derivative at order $n_t$ with respect to $\mu$  on both side of \eqref{c1b}, we get
	$$b_{1,1}(t,x)\partial_\mu^{n_t} h_1(t,\mu)=b_{2,1}(t,x)\partial_\mu^{n_t} h_2(t,\mu+\varphi(t,x))=b_{2,1}(t,x)\partial_\mu^{n_t} h_1(t,\mu+\varphi(t,x)) ,$$
	for $(x,\mu)\in \Omega\times\R$. Fixing $\mu=\mu_t+\varphi(t,x)$ and applying \eqref{c4b}, we get
	$$b_{1,1}(t,x)\partial_\mu^{n_t} h_1(t,\mu_t-\varphi(t,x))=b_{2,1}(t,x)\partial_\mu^{n_t} h_1(t,\mu_t)=0 ,\quad x\in \Omega$$
	and \eqref{c4c} implies
	$$\partial_\mu^{n_t} h_1(t,\mu_t-\varphi(t,x))=0 ,\quad x\in \Omega.$$
	On the other hand, since $\R\ni\mu\mapsto\partial_\mu^{n_t} h_1(t,\mu)$ is analytic either it is uniformly vanishing or its zeros are isolated. By \eqref{c4b} we have that $\partial_\mu^{n_t} h_1(t,\ccdot)\not\equiv0$ for  $t\in(0,T)$. Thus the zeros of $\R\ni\mu\mapsto\partial_\mu^{n_t} h_1(t,\mu)$ are isolated. Using the fact that $\overline{\Omega}\ni x\mapsto\varphi(t,x)$ is continuous, 
	we deduce 
	that the map $\overline{\Omega}\ni x\mapsto\varphi(t,x)$ is constant. 
	Then, recalling that $\varphi(t,x)=0$ for $x\in\partial\Omega$, we deduce that $\varphi(t,\ccdot)\equiv0$. Since here $t\in(0,T)$ is arbitrary chosen we deduce that $\varphi\equiv0$ and it follows that $b_1=b_2$.
	The last statement of the theorem can be proved with similar arguments. 
\end{proof}

\begin{proof}[Proof of Theorem \ref{c5}]
	We will only consider the other statement of the theorem since the last statement can be deduced from similar arguments. Namely, we will prove  that \eqref{t1c} and the conditions of Theorem \ref{t1} imply that $b_1=b_2$. We start by showing that $\partial_\mu h_1=\partial_\mu h_2$. For this purpose, combining \eqref{gauge1} and \eqref{c1b} with \eqref{c5a} we obtain
	$$b_{1,1}(t,x)h_1(t,b_{1,2}(t,x)\mu)+b_{1,0}(t,x)=b_{2,1}(t,x)h_2(t,b_{2,2}(t,x)\mu)+b_{2,0}(t,x)+\mathcal A(t)  \varphi(t,x),$$
	for $(t,x,\mu)\in\Sigma\times\R$.
	Differentiating both sides of this identity with respect to $\mu$, we get
	$$b_{1,2}(t,x)b_{1,1}(t,x)\partial_\mu h_1(t,b_{1,2}(t,x)\mu)=b_{2,2}(t,x)b_{2,1}(t,x)\partial_\mu h_2(t,b_{2,2}(t,x)\mu) ,\quad (t,x,\mu)\in\Sigma\times\R$$
	and \eqref{c5d} implies that
	$$b_{1,1}(t,x_t)b_{1,2}(t,x_t)(\partial_\mu h_1(t,b_{1,2}(t,x_t)\mu)-\partial_\mu h_2(t,b_{1,2}(t,x_t)\mu))=0 ,\quad (t,\mu)\in(0,T)\times\R,$$
	with $b_{1,1}(t,x_t)\neq0$ and $b_{1,2}(t,x_t)\neq0$, $t\in(0,T)$. It follows that $\partial_\mu h_1=\partial_\mu h_2$.
	
	Now let us show that the function $\varphi$ of \eqref{c1b} is identically zero. Fixing $t\in(0,T)$ and applying the derivative at order $n_t$ with respect to $\mu$  on both side of \eqref{c1b}, we get
	$$\begin{aligned}&b_{1,1}(t,x)(b_{1,2}(t,x))^{n_t}\partial_\mu^{n_t} h_1(t,b_{1,2}(t,x)\mu)\\
		&= b_{2,1}(t,x))(b_{2,2}(t,x))^{n_t}\partial_\mu^{n_t} h_2(t,b_{2,2}(t,x)(\mu+\varphi(t,x)))\\
		&= b_{2,1}(t,x)(b_{2,2}(t,x))^{n_t}\partial_\mu^{n_t} h_1(t,b_{2,2}(t,x)(\mu+\varphi(t,x))) ,\quad (x,\mu)\in \Omega\times\R.\end{aligned}$$
	Fixing $\mu=-\varphi(t,x)$ and applying \eqref{c5d}, we get
	$$b_{1,1}(t,x)(b_{1,2}(t,x))^{n_t}\partial_\mu^{n_t} h_1(t,-b_{1,2}(t,x)\varphi(t,x))=b_{2,1}(t,x)(b_{2,2}(t,x))^{n_t}\partial_\mu^{n_t} h_1(t,0)=0 ,\quad x\in \Omega$$
	and \eqref{c5c} implies
	$$\partial_\mu^{n_t} h_1(t,-b_{1,2}(t,x)\varphi(t,x)))=0 ,\quad x\in \Omega.$$
	On the other hand, since $\R\ni\mu\mapsto\partial_\mu^{n_t} h_1(t,\mu)$ is analytic and $\overline{\Omega}\ni x\mapsto b_{1,2}(t,x)\varphi(t,x)$ is continuous, we deduce that  the map $\overline{\Omega}\ni x\mapsto b_{1,2}(t,x)\varphi(t,x)$ is constant. In view of \eqref{c5b}, we can conclude that  $\partial_\mu^{n_t} h_1(t,\ccdot)\not\equiv0$ and the map $\overline{\Omega}\ni x\mapsto b_{1,2}(t,x)\varphi(t,x)$ is constant. Then, recalling that $\varphi(t,x)=0$ for $x\in\partial\Omega$ and applying \eqref{c5c}, we deduce that $\varphi(t,\ccdot)\equiv0$. Since here $t\in(0,T)$ is arbitrary chosen we deduce that $\varphi\equiv0$ and it follows that $b_1=b_2$.
\end{proof}

\begin{proof}[Proof of Corollary \ref{c7}]
	Again, we will only consider the first statement of the corollary as the other statement follows similarly. Namely, we will prove  that \eqref{t1c} and the conditions of Theorem \ref{t1} imply that $b_1=b_2$. For this purpose, we only need to prove that that the function $\varphi$ of \eqref{c1b} is identically zero. We start by  assuming that condition (i) is fulfilled. Fixing $x\in\Omega$ and applying the derivative at order $n_x$ with respect to $\mu$  on both side of \eqref{c1b}, we get
	$$\begin{aligned}&b_{1,1}(t,x)(b_{1,2}(t,x))^{n_x}\partial_\mu^{n_x} G(x,b_{1,2}(t,x)\mu)\\
		&= b_{2,1}(t,x)(b_{2,2}(t,x))^{n_x}\partial_\mu^{n_x} G(x,b_{2,2}(t,x)(\mu+\varphi(t,x)))\\
		&= b_{2,1}(t,x)(b_{2,2}(t,x))^{n_x}\partial_\mu^{n_x} G(x,b_{2,2}(t,x)(\mu+\varphi(t,x)))\\
		&= b_{2,1}(t,x)(b_{1,2}(t,x))^{n_x}\partial_\mu^{n_x} G(x,b_{1,2}(t,x)(\mu+\varphi(t,x))) ,\quad (t,\mu)\in (0,T)\times\R.\end{aligned}$$
	Applying \eqref{c5c}, fixing $\mu=\frac{\mu_x}{b_{1,2}(t,x)}-\varphi(t,x)$ and using \eqref{c7b}, we get
	$$b_{1,1}(t,x)(b_{1,2}(t,x))^{n_x}\partial_\mu^{n_x} G(x,\mu_x+b_{1,2}(t,x)\varphi(t,x))=b_{2,1}(t,x)(b_{1,2}(t,x))^{n_x}\partial_\mu^{n_x} G(x,\mu_x)=0 ,$$
	for $t\in(0,T)$.
	Then, \eqref{c5c} implies
	$$ \partial_\mu^{n_x} G(x,\mu_x+b_{1,2}(t,x)\varphi(t,x))=0 ,\quad t\in (0,T).$$
	On the other hand, since $\R\ni\mu\mapsto\partial_\mu^{n_x} G(x,\mu)$ is analytic and $[0,T]\ni t\mapsto b_{1,2}(t,x)\varphi(t,x)$ is continuous, we deduce that  either  $\partial_\mu^{n_x} G(x,\ccdot)\equiv0$ or  that the map $[0,T]\ni t\mapsto b_{1,2}(t,x)\varphi(t,x)$ is constant. Since  $\partial_\mu^{n_x} G(x,\ccdot)\not\equiv0$ we deduce that the map $[0,T]\ni t\mapsto b_{1,2}(t,x)\varphi(t,x)$ is constant. Then, recalling that $\varphi(0,\ccdot)=0$  and applying \eqref{c5c}, we deduce that, for all $t\in[0,T]$, $\varphi(t,x)=0$. Since here $x\in\Omega$ is arbitrary chosen we deduce that $\varphi\equiv0$ and it follows that $b_1=b_2$.
	
	Combining the above argumentation and the arguments used for the proof of Theorem \ref{c5}, one can easily check that \eqref{t1c} implies also that $b_1=b_2$ when  condition (ii) is fulfilled. This completes the proof of the corollary.
\end{proof}

\begin{proof}[Proofs of Corollary \ref{cor: unique_poten_partial}]
	With the conclusion of Theorem \ref{t5} at hand, combined with $b_j(t,x,\mu)=q_j(t,x)\mu$ for $(t,x,\mu)\in Q\times \R$ and $j=1,2$, both \eqref{t1d} and \eqref{t1d_new} yield that 
	\begin{align}\label{unique of potentials}
		q_1(x,t)\mu=S_\varphi (q_2(x,t)\mu) 
		=q_2(t,x)(\mu+\varphi(t,x))+\rho(t,x) \partial_t \varphi(t,x)+\mathcal A(t)  \varphi(t,x),
	\end{align}
	for any $(t,x,\mu)\in Q\times \R$. In particular, plugging $\mu=0$ into \eqref{unique of potentials}, the function $\varphi$ satisfies the IBVP
	\begin{align*}
		\begin{cases}
			\rho(t,x) \partial_t \varphi(t,x)+\mathcal A(t)  \varphi(t,x)+q_2(t,x)\varphi(t,x)=0 & \text{ in }Q, \\
			\varphi(t,x)=0  & \text{ on }\Sigma,\\
			\varphi(0,x)=0 & \text{ in }\Omega.
		\end{cases}
	\end{align*}
	By the uniqueness of the above IBVP, we must have $\varphi\equiv0$ in $Q$.
	Now, by using \eqref{unique of potentials} again, we have $q_1(t,x)\mu=q_2(t,x)\mu$, for all $\mu \in \R$, which implies $q_1=q_2$ as desired. This completes the proof.
\end{proof}

\section{Application to the simultaneous determination of nonlinear and source terms}\label{sec_simul}

One of the important application of our results is to inverse source problems, where the aim is to recover both the source and nonlinear terms simultaneously. In this section, we consider the following IBVP 
\begin{align}\label{eq100}
	\begin{cases}
		\rho(t,x) \partial_t u(t,x)+\mathcal A(t)  u(t,x)+ d(t,x,u(t,x))  =  F(t,x) & \text{ in } Q,\\
		u(t,x) = f(t,x) & \text{ on } \Sigma, \\  
		u(0,x)  =  0 & \text{ for }x \in \Omega,
	\end{cases}
\end{align}
with $d\in \mathbb A(\R;C^{\frac{\alpha}{2},\alpha}([0,T]\times\overline{\Omega}))$ and $F\in C^{\frac{\alpha}{2},\alpha}([0,T]\times\overline{\Omega})$ satisfying the conditions 
\begin{align}\label{F}
	F(0,x)=0,\quad x\in\partial\Omega,
\end{align}
\begin{align}\label{d}
	d(t,x,0)=0,\quad (t,x)\in Q.
\end{align}
The latter condition is just for presentational purposes and can be removed by redefining $F$.
In a similar way to the problem studied above we assume here that there exists $f=f_0\in\mathcal K_0$ such that \eqref{eq100} admits a unique solution for $f=f_0$. Then, applying Proposition \ref{p1}, we can prove  that there exists $\epsilon>0$, depending on  $a$, $\rho$, $d$, $F$, $f_0$, $\Omega$, $T$, such that, for all $f\in \mathbb B(f_0,\epsilon)$,  \eqref{eq100} admits a unique solution $u_f\in  C^{1+\frac{\alpha}{2},2+\alpha}([0,T]\times\overline{\Omega})$ that lies in a sufficiently small neighborhood of the solution $u_0$ of \eqref{eq100}  when $f=f_0$.
Using these  properties, we can define the parabolic DN map
$$\mathcal M_{(d,F)}:\mathbb B(f_0,\epsilon)\ni f\mapsto \left. \partial_{\nu(a)} u(t,x)\right|_{\Sigma},$$
where $u$ solves \eqref{eq100}.

We consider in this section the inverse problem of determining simultaneously the nonlinear term $d$ and the source term $F$ appearing in \eqref{eq100}. Similarly to the problem (IP), there will be a gauge invariance for this inverse problem. Indeed, fix $\varphi\in   C^{1+\frac{\alpha}{2},2+\alpha}([0,T]\times\overline{\Omega})$ satisfying \eqref{gauge1} and consider the map $U_\varphi$ mapping $ C^\infty(\R;C^{\frac{\alpha}{2},\alpha}([0,T]\times\overline{\Omega}))\times C^{\frac{\alpha}{2},\alpha}([0,T]\times\overline{\Omega})$ into itself and
defined by $U_\varphi (d,F)=(d_\varphi,F_\varphi)$ with 
\begin{align}\label{gauge3}
	\begin{split}
		d_\varphi(t,x,\mu)=&d(t,x,\mu+\varphi(t,x))-d(t,x,\varphi(t,x)),\\
		F_\varphi(t,x)=&F(t,x)-\rho(t,x) \partial_t \varphi(t,x)-\mathcal A(t)  \varphi(t,x)- d(t,x,\varphi(t,x)),
	\end{split}
\end{align}
for $ (t,x,\mu)\in Q\times\R$.
Then, one can easily check that $\mathcal M_{(d,F)}=\mathcal M_{U_\varphi(d,F)}$. Note that \eqref{gauge3} is equivalent to \eqref{gauge2} for
$$b(t,x,\mu)=d(t,x,\mu)-F(t,x),\quad (t,x,\mu)\in Q\times\R.$$
Following, this property, in general the best one can expect for our inverse problem is the determination of  $U_\varphi(d,F)$ from $\mathcal M_{(d,F)}$ for some $\varphi\in   C^{1+\frac{\alpha}{2},2+\alpha}([0,T]\times\overline{\Omega})$ satisfying \eqref{gauge1}. Our first result will be stated in that sense.

\begin{prop}\label{p6} Let $a:=(a_{ik})_{1 \leq i,k \leq n} \in C^\infty([0,T]\times\overline{\Omega};\R^{n \times n})$ satisfy \eqref{ell},  
	$\rho \in C^\infty([0,T]\times\overline{\Omega};\R_+)$ and, for $j=1,2$, let $d_j\in\mathbb A(\R;C^{\frac{\alpha}{2},\alpha}([0,T]\times\overline{\Omega}))\cap C^\infty([0,T]\times\overline{\Omega}\times\R)$ and $F_j\in C^\infty([0,T]\times\overline{\Omega})$ satisfy \eqref{F}-\eqref{d}  with $d=d_j$ and $F=F_j$. We assume also that there exists $f_0\in \mathcal K_0$ such that problem \eqref{eq100}, with $f=f_0$ and $d=d_j$, $F=F_j$, $j=1,2$, is  well-posed.
	Then,  the condition
	\begin{align}\label{p6a}
		\mathcal M_{(d_1,F_1)}=\mathcal M_{(d_2,F_2)} 
	\end{align}
	implies that there exists $\varphi\in   C^{1+\frac{\alpha}{2},2+\alpha}([0,T]\times\overline{\Omega})$ satisfying \eqref{gauge1} such that
	\begin{align}\label{p6b}
		(d_1,F_1)=U_\varphi(d_2,F_2),
	\end{align}
	where $U_\varphi$ is the map defined by \eqref{gauge3}.
\end{prop}
\begin{proof} Fixing 
	$$b_j(t,x,\mu)=d_j(t,x,\mu)-F_j(t,x),\quad (t,x,\mu)\in Q\times\R$$
	one can check that $\mathcal N_{b_j}=\mathcal M_{(d_j,F_j)}$ and \eqref{p6a} implies \eqref{t1c}. Then, applying Theorem \ref{t1}, we deduce that \eqref{t1d} holds true which clearly implies \eqref{p6b}.\end{proof}

It is well known that when $\mu\mapsto d(t,x,\mu)$ is linear, $(t,x)\in Q$, there is no hope to determine general class of source terms $F\in C^\infty([0,T]\times\overline{\Omega})$ satisfying the condition of Proposition \ref{p6} from the knowledge of the map $\mathcal M_{(d,F)}$, see Example \ref{rmk: example} or e.g. \cite[Appendix A]{KSXY}. Nevertheless, this invariance breaks for several class of nonlinear terms $d$ for which we can prove the simultaneous determination of $d$ and $F$ from $\mathcal M_{(d,F)}$. More precisely, applying Corollary \ref{c1} and Theorems \ref{c3}, \ref{c4}, \ref{c5}, we can show the following.

\begin{corollary}\label{c6} Let the condition of Proposition \ref{p6} be fulfilled and assume that, for $j=1,2$, the nonlinear term $d_j$ satisfies one of the following conditions:
	\begin{itemize}
		\item[(i)] There exists $\kappa\in C^{\frac{\alpha}{2},\alpha}([0,T]\times\overline{\Omega})$ such that 
		$$ d_1(t,x,\kappa(t,x))-d_2(t,x,\kappa(t,x))=F_1(t,x)-F_2(t,x),\quad (t,x)\in [0,T]\times\overline{\Omega}.$$
		
		\item[(ii)] There exists $N_j\geq2$ such that
		$$d_j(t,x,\mu)=\sum_{k=1}^{N_j} d_{j,k}(t,x)\mu^k,\quad (t,x,\mu)\in[0,T]\times\overline{\Omega}\times\R,\ j=1,2.$$
		Moreover, for $N=\min(N_1,N_2)$ and $J$ a dense subset of $Q$, we have
		$$\min \left(|(d_{1,N-1}-d_{2,N-1})(t,x)|,\, \sum_{j=1}^2|(d_{j,N}-d_{j,N-1})(t,x)|\right)=0,\ (t,x)\in J,$$
		$$d_{1,N}(t,x)\neq0,\quad (t,x)\in J.$$
		
		\item[(iii)] There exists $h_j\in \mathbb A(\R;C^{\frac{\alpha}{2}}([0,T]))$ such that
		$$d_j(t,x,\mu)=q_{j}(t,x)h_j(t,\mu),\quad (t,x,\mu)\in[0,T]\times\overline{\Omega}\times\R,\ j=1,2.$$
		Assume also that,  for all $t\in(0,T)$,  there exist $\mu_t\in\R$ and $n_t\in\mathbb N$ such that
		$$\partial_\mu^{n_t} h_1(t,\ccdot)\not\equiv0,\quad\partial_\mu^{n_t} h_1(t,\mu_t)=0,\quad t\in(0,T).$$
		Moreover, we assume that 
		$$  q_{1}(t,x)\neq0,\quad  (t,x)\in Q$$
		and that for all $t\in(0,T)$ there exists $x_t\in\partial\Omega$ such that
		$$ q_{1}(t,x_t)=q_{2}(t,x_t)\neq0,\quad  t\in(0,T).$$
		
		\item[(iv)] There exists $h_j(t,\cdot)\in \mathbb A(\R;C^{\frac{\alpha}{2}}([0,T]))$ such that
		$$d_j(t,x,\mu)=q_{j,1}(t,x)h_j(t,q_{j,2}(t,x)\mu),\quad (t,x,\mu)\in[0,T]\times\overline{\Omega}\times\R,\ j=1,2.$$
		Assume also that,  for all $t\in(0,T)$,  there exists  $n_t\in\mathbb N$ such that
		$$\partial_\mu^{n_t} h_1(t,\ccdot)\not\equiv0,\quad\partial_\mu^{n_t} h_1(t,0)=0,\quad t\in(0,T).$$
		Moreover, we assume that 
		$$  q_{1,1}(t,x)\neq0,\quad q_{1,2}(t,x)\neq0\quad  (t,x)\in Q$$
		and that for all $t\in(0,T)$ there exists $x_t\in\partial\Omega$ such that
		$$ q_{1,1}(t,x_t)=q_{2,1}(t,x_t)\neq0,\quad q_{1,2}(t,x_t)=q_{2,2}(t,x_t)\neq0,\quad  t\in(0,T).$$
	\end{itemize}

	Then \eqref{p6a} implies that $d_1=d_2$ and $F_1=F_2$.
\end{corollary}
\begin{proof} Fixing 
	$$b_j(t,x,\mu)=d_j(t,x,\mu)-F_j(t,x),\quad (t,x,\mu)\in Q\times\R$$
	one can check that $\mathcal N_{b_j}=\mathcal M_{(d_j,F_j)}$ and \eqref{p6a} implies \eqref{t1c}. Then, applying Corollary \ref{c1} and Theorem \ref{c3}, \ref{c4}, \ref{c5}, we deduce that for semilinear terms $d_j$ satisfying one of the conditions (i), (ii), (iii), (iv) we have
	$$d_1(t,x,\mu)-F_1(t,x)=b_1(t,x,\mu)=b_2(t,x,\mu)=d_2(t,x,\mu)-F_2(t,x),\quad (t,x,\mu)\in Q\times\R.$$
	Choosing $\mu=0$ in the above identity and applying \eqref{d}, we deduce that $F_1=F_2$ and then $d_1=d_2$.\end{proof}

\appendix

\section{Carleman estimates}\label{sec_app}

In the end of this paper, we prove the Carleman estimate in Section \ref{sec_proof thm 1.2}. For the sake of convenience, we also state the result as follows.

\begin{lem} Let $n\geq3$, $q\in L^\infty(Q)$ and $v\in H^1(Q)\cap L^2(0,T;H^2(\Omega))$ satisfy the condition 
	\begin{equation}\label{eq:v_bndr_cond_proof}
		v|_{\Sigma}=0,\quad v|_{t=0}=0.
	\end{equation}
	Then, there exists $\tau_0>0$ depending on $T$, $\Omega$ and $\norm{q}_{L^\infty(Q)}$ such that  for all $\tau>\tau_0$ the following estimate 
	\begin{align}\label{eq:carl_Estim_proof} \begin{split}
			&\tau\int_0^T\int_{\Gamma_{+}(x_0)}e^{-2(\tau^2t+\tau\psi(x))}\abs{\partial_\nu v}^2\abs{\partial_\nu\psi(x) } d\sigma(x)dt+\tau^2\int_Qe^{-2(\tau^2t+\tau\psi(x))}\abs{v}^2dxdt\\
			\leq & C\left(\int_Qe^{-2(\tau^2t+\tau\psi(x))}\abs{(\partial_t-\Delta_x+q)v}^2\, dxdt\right. \\
			& \qquad \left. +\tau\int_0^T\int_{\Gamma_{-}(x_0)}e^{-2(\tau^2t+\tau\psi(x))}\abs{\partial_\nu v}^2\abs{\partial_\nu\psi(x) }d\sigma(x)\, dt\right)
	\end{split}\end{align}
	holds true.
\end{lem}

\begin{proof}
	Recall that $\psi(x)=|x-x_0|$, $x\in\Omega$, for $x_0\in\R^n\setminus\overline{\Omega}$, then $\psi$ satisfies the eikonal equation $\left|\nabla_x \psi(x)\right|=1$   for $x\in \Omega$. 
	Without loss of generality we assume that $u$ is real valued and $q=0$.
	In order to prove the estimate \eqref{eq:carl_Estim_proof}, we fix $v\in C^2(\overline{Q})$ satisfying \eqref{eq:v_bndr_cond_proof}, $s>0$ and  we set 
	$$
	w=e^{-(\tau^2t+\tau\psi(x)-s\frac{\psi(x)^2}{2})}v
	$$   
	in such a way that
	\begin{equation}\label{l7c}e^{-(\tau^2t+\tau\psi(x)-s\frac{\psi(x)^2}{2})}(\pd_t-\Delta_x) v=P_{\tau,s}w,\end{equation}
	where   $P_{\tau,s}$ is given by
	$$P_{\tau,s}=\pd_t-\Delta +2s\tau\psi-\tau \Delta \psi+s\psi\Delta\psi-s^2\psi^2+s-2\tau\nabla\psi\cdot\nabla+2s\psi \nabla\psi\cdot\nabla.$$
	Here we used the fact that the function $\psi$ satisfies the eikonal equation.

	We next decompose $P_{\tau,s}$ into two parts $P_{\tau,s}=P_{\tau,s,+}+P_{\tau,s,-}$ with
	\begin{align*}
		P_{\tau,s,+}:=&-\Delta +2s\tau\psi-\tau \Delta \psi +s\psi \Delta\psi-s^2\psi^2+s   ,\\
		P_{\tau,s,-}:=&\pd_t-2\tau\nabla\psi\cdot\nabla+2s\psi \nabla\psi\cdot\nabla.
	\end{align*}
	Then, it follows that
	\begin{equation}\label{l7d} 
		\begin{split}
			\norm{P_{\tau,s}w}_{L^2(Q)}^2\geq &2\int_Q \left( P_{\tau,s,+}w\right) \left( P_{\tau,s,-}w\right) dxdt\\
			:=&I+II+III+IV+V+VI+VII,
		\end{split}
	\end{equation}
	where  
	$$I=-2\int_Q\pd_tw\Delta w\, dxdt,\quad II=4\tau\int_Q\Delta w\nabla\psi\cdot\nabla w\, dxdt,$$
	$$III=4s\tau\int_Q\psi w \pd_tw \, dxdt,\quad IV=-8s\tau^2\int_Q\psi w\nabla\psi\cdot\nabla w\, dxdt,$$ 
	$$
	V=-4s \int_{Q} \Delta w \psi \nabla \psi \cdot \nabla w \, dxdt, \quad VI=8s^2 \tau\int_{Q} \psi^2  w\nabla \psi \cdot \nabla w \, dxdt,
	$$
	and 
	$$VII= 2\int_Q\left[-\tau\Delta \psi+s\psi\Delta\psi-s^2\psi^2+s\right]w \left( P_{\tau,s,-}w\right) dxdt.$$

	Recalling that $w|_\Sigma=0$ and $w|_{t=0}=0$, fixing 
	$$c_*=\inf_{x\in\Omega}\psi(x)>0$$
	and integrating by parts, we find 
	\begin{align*}
		\begin{split}
			I&=2\int_Q\pd_t\nabla w\cdot\nabla w\, dxdt=\int_Q\pd_t|\nabla w|^2\, dxdt=\int_\Omega |\nabla w(T,x)|^2\, dx\geq0, \\
			III&=2s\tau\int_Q\psi\pd_t(w^2)\, dxdt=2s\tau\int_\Omega \psi(x) w(T,x)^2\, dx\geq 2c_*s\tau\int_\Omega  w(T,x)^2\, dx\geq0, \\
			IV&=-4s\tau^2\int_Q(x-x_0)\cdot\nabla (w^2)\,dxdt =4s\tau^2\int_Q\textrm{div}(x-x_0) w^2\, dxdt\\
			&=4ns\tau^2\int_Qw^2\, dxdt\geq0,
		\end{split}
	\end{align*}
	and similarly,
	$$
	VI=-4s^2 \tau \int_{Q} \textrm{div}(\psi^2 \nabla \psi) w^2 \, dxdt = -4(n+1)s^2 \tau \int_Q |x-x_0| w^2 \, dxdt,
	$$
	where we utilized the fact that $\psi^2 \nabla \psi = |x-x_0|(x-x_0)$. 

	Now, let us consider $II$. Integrating by parts and using the fact that $w|_\Sigma=0$, we get
	$$\begin{aligned}II&=4\tau\int_\Sigma \pd_\nu w\nabla\psi\cdot\nabla w\, d\sigma(x)dt-4\tau\int_Q\nabla w\cdot\nabla(\nabla\psi\cdot\nabla w)\, dxdt\\
		&=4\tau\int_\Sigma (\pd_\nu w)^2\pd_\nu\psi \, d\sigma(x)dt-4\tau\int_Q D^2\psi(\nabla w,\nabla w)\, dxdt-2\tau \int_Q\nabla\psi\cdot\nabla (|\nabla w|^2)\, dxdt\\
		&=2\tau\int_\Sigma (\pd_\nu w)^2\pd_\nu\psi \, d\sigma(x) dt-4\tau\int_Q D^2\psi(\nabla w,\nabla w)\, dxdt+2\tau \int_Q\Delta\psi|\nabla w|^2\, dxdt\\
		&=2\tau\int_\Sigma (\pd_\nu w)^2\pd_\nu\psi \, d\sigma(x)dt-4\tau\int_Q D^2\psi(\nabla w,\nabla w)\, dxdt+2\tau \int_Q\frac{n-1}{|x-x_0|}|\nabla w|^2\, dxdt.\end{aligned}$$
	On the other hand, one can check that $$D^2\psi(x)=|x-x_0|^{-3}\left(|x-x_0|^2\mathrm{Id}_{\R^{n\times n}}-N(x)\right),\quad x\in\Omega$$
	with $N(x)=\left((x_i-x_0^i)(x_j-x_0^j)\right)_{1\leq i,j\leq n}$ where $x=(x_1,\ldots,x_n)$ and $x_0=(x_0^1,\ldots,x_0^n)$. Moreover, it can be proved that $N(x)$ is a symmetric matrix whose eigenvalues are either $0$ or $|x-x_0|^2$. 
	Thus, we get
	$$0\leq D^2\psi(x)(\nabla w(t,x),\nabla w(t,x))\leq \frac{|\nabla w(t,x)|^2}{|x-x_0|},\quad \text{ for }(t,x)\in Q$$
	and it follows that
	\begin{align}\label{II estimate}
		\begin{split}
			II&\geq 2\tau\int_\Sigma (\pd_\nu w)^2\pd_\nu\psi \, d\sigma(x)dt+\tau \int_Q\frac{2(n-3)}{|x-x_0|}|\nabla w|^2\, dxdt\\
			&\geq 2\tau\int_\Sigma (\pd_\nu w)^2\pd_\nu\psi \, d\sigma(x)+c_1(n-3)\tau \int_Q|\nabla w|^2\, dxdt,
		\end{split}
	\end{align}
	where $c_1=\inf_{x\in\Omega}\frac{2}{|x-x_0|}$.

	Applying similar computations for $V$, one has that 
	\begin{align}\label{V estimate}
		\begin{split}
			V=&-2s \int_{Q} \Delta w  \nabla \psi^2 \cdot \nabla w \, dxdt \\
			=&-2s\int_\Sigma \pd_\nu w\nabla\psi^2\cdot\nabla w\, d\sigma(x)dt +2s\int_Q\nabla w\cdot\nabla(\nabla\psi^2\cdot\nabla w)\, dxdt\\
			=&-2s\int_\Sigma (\pd_\nu w)^2\pd_\nu\psi^2 \, d\sigma(x)dt+2s\int_Q D^2\psi^2(\nabla w,\nabla w)\, dxdt +s\int_Q\nabla\psi^2\cdot\nabla (|\nabla w|^2)\, dxdt\\
			=&-s\int_\Sigma (\pd_\nu w)^2\pd_\nu\psi^2 \, d\sigma(x) dt+2s\int_Q D^2\psi^2(\nabla w,\nabla w)\, dxdt- s \int_Q\Delta\psi^2|\nabla w|^2\, dxdt\\
			=&-s\int_\Sigma (\pd_\nu w)^2\pd_\nu\psi^2 \, d\sigma(x)dt+2s\int_Q D^2\psi^2 (\nabla w,\nabla w)\, dxdt-sn\int_Q|\nabla w|^2\, dxdt,\\
			\geq &\underbrace{-s\int_0^T\int_{\Gamma_{+}(x_0)} 2\psi(\pd_\nu w)^2\pd_\nu\psi \, d\sigma(x)dt+s(4-n)\int_Q|\nabla w|^2\, dxdt}_{\text{Here we use }\Delta \psi^2 = \Delta |x-x_0|^2 =n \text{ and }D^2\psi^2(\nabla w, \nabla w) = 2 |\nabla w|^2.}.
		\end{split}
	\end{align} 
	In addition, the sum of the last two terms in the right hand side of \eqref{II estimate} and \eqref{V estimate} is 
	\begin{align}\label{II+V grad estimate}
		\begin{cases}
			s\int_Q|\nabla w|^2\, dxdt  &\text{ for }n=3,\\
			c_1(n-3)\tau \int_Q|\nabla w|^2\, dxdt+s(4-n)\int_Q|\nabla w|^2\, dxdt & \text{ for }n\geq 4.
		\end{cases}
	\end{align}
	By choosing $\tau  \geq \frac{s(4-n)}{c_1(n-3)}$ for $n\geq 4$, both cases appear in \eqref{II+V grad estimate} are nonnegative.   
	Now fixing $c_2=2\sup_{x\in\partial\Omega}\psi(x)$, and choosing 
	\begin{align*}
		\tau >\tau_1(s):=\begin{cases}
			c_2 s &\text{ for }n=3,\\
			\max\left( c_2s,\frac{s(4-n)}{c_1(n-3)} \right) &\text{ for }n \geq 4,
		\end{cases}
	\end{align*}
	we can deduce that
	$$II+V\geq \tau\int_0^T\int_{\Gamma_{+}(x_0)} (\pd_\nu w)^2|\pd_\nu\psi| \, d\sigma(x)dt-2\tau\int_0^T\int_{\Gamma_{-}(x_0)} (\pd_\nu w)^2|\pd_\nu\psi| \, d\sigma(x)dt .$$

	In addition, repeating the above argumentation, it is not hard to check that there exists a constant $c_0>0$ independent of $s$ and $\tau$ such that
	$$VI+VII\geq -c_0\left((\tau^2+s^2\tau+s^2)\int_Qw^2\, dxdt+(\tau+s^2)\int_\Omega  w(T,x)^2\, dx\right).$$
	Thus, choosing $s=c_0(1+c_*^{-1})+1$, $\tau_0=\max\left( \tau_1(s),\frac{sn}{c_1},3s^2+1, \frac{sc_0}{c_*}\right)$ and applying the above estimates, for all $\tau>\tau_0$, we obtain
	\begin{multline*}
		\norm{P_{\tau,s}w}_{L^2(Q)}^2
		\geq2\int_QP_{\tau,s,+}wP_{\tau,s,-}w\, dxdt
		\\
		\geq c_1\tau^2\int_Qw^2\, dxdt+\tau\int_0^T\int_{\Gamma_{+}(x_0)} (\pd_\nu w)^2|\pd_\nu\psi| \, d\sigma(x)dt-2\tau\int_0^T\int_{\Gamma_{-}(x_0)} (\pd_\nu w)^2|\pd_\nu\psi| \, d\sigma(x)dt,
	\end{multline*}
	with $c_1>0$ depending only on $\Omega$. From this last estimate and the fact that 
	$$\pd_\nu\psi(x)=\frac{(x-x_0)\cdot\nu}{|x-x_0|},\quad x\in\partial\Omega,$$
	we deduce easily \eqref{eq:carl_Estim_proof}.\end{proof}

\medskip 

\noindent\textbf{Acknowledgment.} T.L. was supported by the Academy of Finland (Centre of
Excellence in Inverse Modeling and Imaging, grant numbers 284715 and 309963). 
Y.-H. Lin is partially supported by the National Science and Technology Council (NSTC) Taiwan, under the projects 111-2628-M-A49-002 and 112-2628-M-A49-003. 

\bibliography{refs} 

\bibliographystyle{alpha}

\end{document}